\documentclass[11pt,reqno]{amsart}
\usepackage[utf8]{inputenc}
\usepackage{mathtools,graphicx,xcolor,float,minipage-marginpar,blindtext,amssymb,amsmath,amsthm,cleveref,enumitem}
\graphicspath{{image/}}
\usepackage{wrapfig}
\usepackage[mathscr]{euscript}
\usepackage[a4paper, hmargin=2cm, vmargin = 3cm]{geometry}

\allowdisplaybreaks[4]
\newcommand\norm[1]{\left\|#1\right\|}

\makeatletter\newsavebox{\@brx}\newcommand{\llangle}[1][]{\savebox{\@brx}{\(\m@th{#1\langle}\)}%
  \mathopen{\copy\@brx\kern-0.5\wd\@brx\usebox{\@brx}}}
\newcommand{\rrangle}[1][]{\savebox{\@brx}{\(\m@th{#1\rangle}\)}%
  \mathclose{\copy\@brx\kern-0.5\wd\@brx\usebox{\@brx}}}
\makeatother

\makeatletter\newsavebox{\@brxx}\newcommand{\lbbrac}[1][]{\savebox{\@brx}{\(\m@th{#1[}\)}%
  \mathopen{\copy\@brx\kern-0.5\wd\@brx\usebox{\@brx}}}
\newcommand{\rbbrac}[1][]{\savebox{\@brx}{\(\m@th{#1]}\)}%
  \mathclose{\copy\@brx\kern-0.5\wd\@brx\usebox{\@brx}}}
\makeatother

\let\oldsqrt\sqrt
\def\sqrt{\mathpalette\DHLhksqrt}
\def\DHLhksqrt#1#2{%
\setbox0=\hbox{$#1\oldsqrt{#2\,}$}\dimen0=\ht0
\advance\dimen0-0.2\ht0
\setbox2=\hbox{\vrule height\ht0 depth -\dimen0}%
{\box0\lower0.4pt\box2}}

\makeatletter
\renewcommand*\env@matrix[1][*\c@MaxMatrixCols c]{%
  \hskip -\arraycolsep
  \let\@ifnextchar\new@ifnextchar
  \array{#1}}
\makeatother

\def\a{\alpha}
\def\b{\beta}
\def\x{\chi}
\def\p{\varphi}
\def\q{\psi}
\def\l{\lambda}
\def\g{\gamma}
\def\r{\rho}
\def\s{\sigma}
\def\d{\delta}
\def\z{\zeta}
\def\th{\theta}
\def\P{\Phi}
\def\Q{\Psi}
\def\D{\Delta}
\def\rmd{\mathrm{d}}
\def\NN{\mathbb{N}}
\def\ZZ{\mathbb{Z}}
\def\RR{\mathbb{R}}
\def\CC{\mathbb{C}}
\def\DD{\mathbb{D}}
\def\BB{\mathbb{B}}
\def\SS{\mathbb{S}}
\def\TT{\mathbb{T}}
\def\HH{\mathbb{H}}
\def\cC{\mathcal{C}}
\def\cZ{\mathcal{Z}}
\def\rms{\mathrm{s}}

\newtheorem{theorem}{Theorem}[section]
\newtheorem{lemma}[theorem]{Lemma}
\newtheorem{corollary}[theorem]{Corollary}

\theoremstyle{definition}
\newtheorem{definition}[theorem]{Definition}
\newtheorem{remark}[theorem]{Remark}

\numberwithin{equation}{section}

\renewcommand{\tilde}{\widetilde}
\renewcommand{\hat}{\widehat}
\renewcommand{\Re}{\mathrm{Re}\,}
\renewcommand{\Im}{\mathrm{Im}}

\usepackage{xpatch}
\makeatletter   
\xpatchcmd{\@tocline}
{\hfil\hbox to\@pnumwidth{\@tocpagenum{#7}}\par}
{\ifnum#1<0\hfill\else\dotfill\fi\hbox to\@pnumwidth{\@tocpagenum{#7}}\par}
{}{}
\makeatother

\makeatletter
 \def\l@subsection{\@tocline{2}{0pt}{4pc}{6pc}{}}
\def\l@subsubsection{\@tocline{3}{0pt}{8pc}{8pc}{}}
 \makeatother

\title[Quaternionic Orthogonal Polynomials, a S--V Theorem \& Baxter's Theorem]{Orthogonal Polynomials, a Szeg\H{o}--Verblunsky Theorem and Baxter's Theorem on the Quaternionic Sphere}
\date{\today}
\author[C. J. Gauntlett]{Connor J. Gauntlett}
\address{(CJG) School of Mathematics and Statistics\\
Newcastle University\\
Newcastle upon Tyne NE1 7RU UK}
\email{c.gauntlett@newcastle.ac.uk}

\author[D. P. Kimsey]{David P. Kimsey}
\address{(DPK) School of Mathematics and Statistics\\
Newcastle University\\
Newcastle upon Tyne NE1 7RU UK}
\email{david.kimsey@newcastle.ac.uk}

\begin{document}

\begin{abstract}
    We introduce a theory of orthogonal polynomials on the unit sphere of the quaternions based on the notion of a $q$-positive measure (which originated in a work of Alpay, Colombo, the second author and Sabadini). The results we extend to this setting include the Szeg\H{o} recurrences, the Zeros Theorem for orthogonal polynomials, the Szeg\H{o}--Verblunsky theorem, and Baxter's theorem; to obtain these results, we utilise the Verblunsky coefficients (or Schur parameters) of Alpay, Colombo and Sabadini and a number of established results in the matricial setting. Our approach also requires matrix-valued analogues of Schur's recurrences for the coefficients of a Schur function and of Verblunsky's formula for the moments of a measure, which appear to be new. 
\end{abstract}

\maketitle

\tableofcontents

\section{Introduction}

This paper generalises many ideas from the theory of orthogonal polynomials on the unit circle, including Szeg\H{o}'s recursion, the Zeros Theorem for orthogonal polynomials on the unit circle, the Szeg\H{o}--Verblunsky theorem and Baxter's theorem, to a quaternionic setting. In this setting, the usual positive measure on the circle is replaced with a \emph{$q$-positive measure}, a quaternionic-valued measure which satisfies a certain positivity property that first appeared in \cite{ACKS15} (see \Cref{def:qPosMeasure} below). This theory is intricately linked with that of matrix-valued orthogonal polynomials on the unit circle, and so we shall often need to additionally study that setting. We also remark that the first results of this paper, which constitute the body of \Cref{sec:SchurRecVerblunskyForm}, may be of independent interest to experts in the field of matrix orthogonal polynomials on the unit circle.

Let us first set the stage with the relevant classical theory of orthogonal polynomials on the unit circle, drawing primarily from \cite{Sim05a}. Let \(\mu\) be a probability measure on the unit circle \(\TT\), and uniquely associate to \(\mu\) a number of objects as follows. Define the \emph{Herglotz function} of \(\mu\) --- a function \(F : \DD \to \CC\) with \(\Re F(z) \geq 0\) for \(z \in \DD\) --- to be the Riesz--Herglotz transform of \(\mu\):
\[
    F(z) := \int_0^{2\pi} \frac{1 + ze^{-i\th}}{1 - ze^{-i\th}} \, \rmd\mu(\th).
\]
Recall that the \emph{moments} of \(\mu\) are
\[
    c_n := \int_0^{2\pi} e^{-in\th} \, \rmd\mu(\th)
\]
for \(n \in \ZZ\); one may show that \(F\) has Taylor series
\[
    F(z) = 1 + 2 \sum_{k=1}^{\infty} c_n z^n.
\]

We obtain from \(F\) the \emph{Schur function} of \(\mu\) by taking the Cayley transform, i.e. the conformal mapping from the right half-plane to the unit disc, of \(F\): \(f : \DD \to \CC\) with \(\lvert f(z) \rvert \leq 1\) for \(z \in \DD\) is such that
\[
    F(z) = \frac{1 + f(z)}{1 - f(z)}.
\]
The \emph{Schur algorithm} is a recursive algorithm which takes a Schur function and outputs a sequence \((f_n)_{n=0}^{\infty}\) of Schur functions, called the \emph{Schur iterates} of the input. If the Schur iterates are Schur functions \((f_n)_{n=0}^{\infty}\) with \(f_0 := f\), then the \emph{Verblunsky coefficients} (also called \emph{Schur parameters}) of \(\mu\) are the elements of the sequence \((\g_n)_{n=0}^{\infty}\) given by \(\g_n := f_n(0)\), and the Schur algorithm obtains \(f_{n+1}\) from \(f_n\) and \(\g_n\) via
\[
    f_{n+1}(z) := \frac{1}{z} \frac{f_n(z) - \g_n}{1 - \overline{\g_n}f_n(z)}.
\]

Schur also proved \cite{Sch86a} the \emph{Schur recurrences} for the Taylor series coefficients of the Schur function \(f\) in terms of the Verblunsky coefficients. Namely, if \(f\) has coefficients \((s_n(f))_{n=0}^{\infty}\), then Schur showed that
\[
    s_n(f) = \g_n \prod_{k=0}^{n-1} (1 - \lvert \g_k \rvert^2) + r_n(\g_0, \ldots, \g_{n-1}, \overline{\g_0}, \cdots, \overline{\g_{n-1}})
\]
for some multivariate polynomials \(r_n\). Schur's recurrences may be used to further see the following:
\[
    c_n = \g_{n-1} \prod_{k=0}^{n-2} (1 - \lvert \g_k\rvert^2) + \tilde{r}_{n-1}(\g_0, \ldots, \g_{n-2}, \overline{\g_0}, \cdots, \overline{\g_{n-2}}),
\]
where \(\tilde{r}_n\) are some multivariate polynomials, distinct from \(r_n\). Simon \cite{Sim05a} attributes this latter formula to Schur in 1917; later, it was independently shown by Verblunsky (see \cite[Theorem 1]{Ver35a}). We remark that Verblunsky did not yet have the connection between the coefficients he was considering and those pioneered by Schur, and thus Verblunsky did not use Schur's recurrences to prove the formula. The proof we refer to here is due to Simon, whose nomenclature we borrow to refer to this formula as \emph{Verblunsky's formula}.

In parallel, one may obtain orthonormal polynomials with respect to \(\mu\), \((\p_n)_{n=0}^{\infty}\), by performing the Gram-Schmidt algorithm on the set of monomials \(\{1, z, z^2, \ldots\}\) ordered in the usual fashion with respect to the inner product given by
\[
    \langle p, q \rangle_\mu := \int_0^{2\pi} p(e^{i\th}) \overline{q(e^{i\th})} \, \rmd\mu(\th).
\]
Of course, this is only a genuine inner product when \(\langle \cdot, \cdot \rangle_\mu\) is a positive definite form, which happens if and only if the underlying measure \(\mu\) is not a finite sum of atomic measures (some sources, including \cite{Sim05a}, refer to this condition as \(\mu\) being \emph{non-trivial}, a term we shall adopt for an analogous condition in the quaternions).

The Verblunsky coefficients also arise as the coefficients of paired recurrences between the orthonormal polynomials and their reverse polynomials. Recall that the reverse polynomial of a polynomial \(p(z) = \sum_{k=0}^{n} p_k z^k\) is the polynomial \(p^{\#}\) with \(k^{\text{th}}\) coefficient \(\overline{p_{n-k}}\), \(k = 0, \ldots, n\), and is given by the formula
\[
    p^{\#}(z) = z^n \overline{p(1/\overline{z})}.
\]
The \emph{Szeg\H{o} recurrences}, first proved by Szeg\H{o} \cite{Sze75}, relate the orthonormal polynomials, their reverse polynomials, and the Verblunsky coefficients via
\[
    \p_{n+1}(z) = \frac{z \p_n(z) - \overline{\g}_n \p_n^{\#}(z)}{\sqrt{1 - \lvert \g_n \rvert^2}}
\]
and
\[
    \p^{\#}_{n+1}(z) = \frac{\p^{\#}(z) - \g_n z \p_n(z)}{\sqrt{1 - \lvert \g_n \rvert^2}}.
\]
It is a theorem of Geronimus \cite{Ger43} that the coefficients appearing in the Szeg\H{o} recurrences do indeed coincide with the output of the Schur algorithm. The Szeg\H{o} recurrences are highly versatile in the theory; for example, they may be used to directly prove the Christoffel--Darboux formula, which relates the orthonormal polynomials and Verblunsky coefficients of a measure to a particular sequence of reproducing kernels.

Another application of the Szeg\H{o} recurrences is in proving the \emph{Zeros Theorem} for orthogonal polynomials: if \((\p_n)_{n=0}^{\infty}\) are the orthonormal polynomials of some (probability) measure on the unit circle, then the zeros of \(\p_n\) lie strictly inside the open unit disc, and as a corollary, the zeros of their reverse polynomials \((\p_n^{\#})_{n=0}^{\infty}\) lie strictly outside the closed unit disc. This result dates back to Geronimus \cite{Ger46}; see also Theorem 1.7.1 of \cite{Sim05a} for a comprehensive evaluation of six different proofs.

One of the headline results in this field, and of particular interest to us, is the Szeg\H{o}--Verblunsky theorem, which goes as follows.
\begin{theorem}[Szeg\H{o}--Verblunsky]
    Let \(\mu\) be a non-finitely atomic probability measure on \(\TT\) with Lebesgue decomposition \(\rmd\mu = w \, \frac{\rmd\th}{2\pi} + \rmd\mu_\rms\) and Verblunsky coefficients \((\g_n)_{n=0}^{\infty}\). Then
    \[
        \prod_{n=0}^{\infty} (1 - \lvert \g_n \rvert^2) = \exp\left( \int_0^{2\pi} \log\big(w(e^{i\th})\big) \, \frac{\rmd\th}{2\pi}\right).
    \]
\end{theorem}
One may apply a well-known lemma of real analysis (see e.g. the remark following Lemma 3.3.5 of \cite{Kat04}) to obtain the following useful consequence of the Szeg\H{o}--Verblunsky theorem, phrased in terms of square-summability of the Verblunsky coefficients:
\[
    \sum_{n=0}^{\infty} \lvert \g_n \rvert^2 < \infty \quad \text{if and only if} \quad \int_0^{2\pi} \log \big(w(e^{i\th})\big) \, \frac{\rmd\th}{2\pi} > -\infty.
\]
Since Verblunsky's theorem \cite{Ver35a} says that any sequence contained within the open unit disc is the sequence of Verblunsky coefficients for some measure on the circle, the Szeg\H{o}--Verblunsky theorem (and its above consequence) may be applied to study square-summability of sequences in \(\DD\) via log-integrability of certain \(L^1(\TT; \rmd\mu)\) functions; conversely, one may study log-integrability via questions of square-summability of sequences in \(\DD\).

A stronger condition one might impose on the Verblunsky coefficients is simply \emph{summability}, i.e. the convergence of the sum
\[
    \sum_{n=0}^\infty \lvert \g_n \rvert.
\]
In this direction, too, a very well-known result exists --- in this case, the result is Baxter's theorem, first proven by Baxter in \cite{Bax61}. The following formulation of Baxter's theorem is taken from \cite{Sim05a}.
\begin{theorem}[Baxter]
    Let \(\mu\) be a non-finitely atomic probability measure on \(\TT\) with Lebesgue decomposition \(\rmd\mu = w \, \frac{\rmd\th}{2\pi} + \rmd\mu_\rms\) and Verblunsky coefficients \((\g_n)_{n=0}^{\infty}\). Then
    \[
        \sum_{n=0}^\infty \lvert \g_n \rvert < \infty
    \]
    if and only if:
    \begin{enumerate}[label={\upshape\roman*)}]
        \item the moments \((c_n)_{n=0}^{\infty}\) of \(\mu\) satisfy
        \[
            \sum_{n=1}^{\infty} \lvert c_n \rvert < \infty,
        \]
        which implies that \(\rmd\mu_\rms = 0\) and \(w\) is continuous on \(\TT\);
        \item \(w(e^{i\th}) > 0\) for all \(\th \in [0,2\pi)\).
    \end{enumerate}
\end{theorem}

Similar to the conclusions we may draw from the Szeg\H{o}--Verblunsky theorem regarding square-summability, Baxter's theorem provides a full characterisation of summable sequences inside \(\DD\) or of continuous functions on \(\TT\) taking strictly positive values; this is more properly viewed as a result about elements of the \emph{Wiener algebra}.

The Wiener algebra (of the circle) is the Banach space \(\mathscr{W}(\TT)\) (or simply \(\mathscr{W}\)) of continuous functions on \(\TT\) whose Fourier coefficients are summable, with respect to the norm $\| w \| := \sum_{n=-\infty}^{\infty} |w_n|$, where $w(e^{i\theta}) = \sum_{n=-\infty}^{\infty} w_n e^{i n \theta}$. Observe that if
\[
    w(e^{i\th}) = \sum_{n=-\infty}^{\infty} c_n e^{in\th}
\]
in \(\mathscr{W}\) is positive, then the absolutely continuous measure given by \(\rmd\mu = w \, \frac{\rmd\th}{2\pi}\) has moments \((c_n)_{n=-\infty}^{\infty}\); in this light, summability of the Wiener function's Fourier coefficients indeed becomes condition i) of Baxter's theorem and summability of Verblunsky coefficients becomes a question of deciding positivity of a particular continuous function on \(\TT\).

Much of the theory of orthogonal polynomials on the unit circle has been generalised to the setting of matrix-valued polynomials. As far as the authors are aware, this began with the seminal paper of Delsarte, Genin and Kamp \cite{DGK78}: this paper begins with a matrix-valued measure on the unit circle and provides, among other things, a matricial analogue of the Szeg\H{o} recurrences, the Christoffel--Darboux formula, and a Szeg\H{o}--Verblunsky theorem involving determinants of their matricial objects. We further mention \cite{DHKT11} for a matrix-valued generalisation of the Szeg\H{o}--Verblunsky theorem, which is formulated in terms of a limit of integrals of matrix-valued orthogonal polynomials and which bears a striking resemblance to the classical Szeg\H{o}--Verblunsky theorem when the Verblunsky coefficients are mutually commuting and normal. 

Shortly following \cite{DGK78}, a matricial Baxter's theorem was established by Geronimo \cite{Ger81}, and more recently, \cite{DK16} established a bijective correspondence, in particular, between matrix-valued spectral densities that appear in Baxter's theorem with CMV matrices and certain solutions of a matricial Nehari problem. For a general and authoritative survey (with numerous new results) on matrix-valued orthogonal polynomials on the unit circle, the reader is strongly encouraged to see \cite{DPS08}; indeed, we shall often refer to this source throughout this paper.

In the quaternionic setting, a notion of regularity or holomorphy, due to Gentili and Struppa, which has proven very useful is given by the notion of slice hyperholomorphic functions, see \cite{GP06} and \cite{GP07}. An authoritative exposition coupled with some applications of slice hyperholomorphic functions and also related spectral theory (based on the $S$-spectrum) and operator theory in the quaternionic and general Clifford algebra settings can be found in \cite{CSS11}. For a presentation of the spectral theorem for (both bounded and unbounded) normal operators on a quaternionic Hilbert space, see \cite{CGK18}. One notable application of slice hyperholomorphic functions is the abstraction of Schur analysis to the quaternionic setting --- for a comprehensive discussion of this theory, see Alpay, Colombo and Sabadini \cite{ACS16a} (and references therein). In particular, \cite{ACS16a} contains a generalisation of the Schur algorithm for coefficient stripping to quaternionic Schur functions (this algorithm first appeared in the paper \cite{ACS16b} by the same authors). We shall use the \(q\)-positive measures of \cite{ACKS15} to generate sequences of slice hyperholomorphic polynomials, orthogonal with respect to particular inner products we shall define.

Certain families of monogenic Clifford algebra-valued polynomials have been studied previously, under various regimes, see, e.g., \cite{CGM01}, \cite{BCG04}, \cite{BCG06}, \cite{BG10}, \cite{BGLS12}, \cite{CFM14}, \cite{CFM17} and more recently \cite{BSZ25}. Monogenic Clifford algebra-valued polynomials can be thought of as a noncommutative generalisation of orthogonal polynomials in several real variables, which lie in the kernel of a certain Dirac operator.

The Wiener algebra, too, has already been studied in the quaternionic setting, in particular in \cite{ACKS16}. We shall introduce the relevant ideas and notation in this area in \Cref{sec:Baxter}.

Finally, we note that the main thrust of this paper is towards a quaternionic (so in particular multidimensional and noncommutative) generalisation of the theory of orthogonal polynomials on the unit circle. The reader may therefore also be interested in previous work of the present authors on a noncommutative generalisation of many of these ideas in the alternative sense of \emph{noncommutative function theory} (\cite{GK25a}) and on a multidimensional analogue in several complex variables (\cite{GK25b}).

\subsection{Main Conclusions}

The main conclusions of this paper are as follows.

\begin{enumerate}[label={(C\arabic*)}]
    \item Given a positive matrix-valued measure, we derive matricial analogues of the Schur recurrences (see \Cref{thm:MatrixSchurRec}) and Verblunsky's formula (see \Cref{thm:MatrixVerForm}) in terms of the matricial Verblunsky coefficients of \cite{DPS08}.

    \item Given a non-trivial \(q\)-positive measure \(\mu_i\) on \(\TT_i\) for some quaternionic unit \(i \in \SS\), we introduce sequences \((\q_n^L)_{n=0}^{\infty}\) and \((\q_n^R)_{n=0}^{\infty}\) of polynomials which are orthonormal with respect to \(\mu_i\) in two appropriate senses, and define their reverse polynomials \((\q_n^{L,\#})_{n=0}^{\infty}\) and \((\q_n^{R,\#})_{n=0}^{\infty}\). We derive for these polynomials Szeg\H{o} recurrences intertwining the four sequences \((\q_n^L)_{n=0}^{\infty}\), \((\q_n^R)_{n=0}^{\infty}\), \((\q_n^{L,\#})_{n=0}^{\infty}\), and \((\q_n^{R,\#})_{n=0}^{\infty}\) (see \Cref{thm:QuaternionicSzegoRecs} and \Cref{cor:PairedSzegoRecs}).

    \item We prove a quaternionic analogue of the Zeros Theorem for the quaternionic orthonormal polynomials associated to a \(q\)-positive measure: namely, the zeros of the orthonormal polynomials of a \(q\)-positive measure lie strictly inside the quaternionic unit ball (see \Cref{thm:QuaternionicZerosThm}).

    \item We prove a quaternionic analogue of the Christoffel--Darboux formula, restricted to the diagonal \(z = \z\) (see \Cref{thm:CDFDiagonal}).

    \item We propose and prove an analogue of the Szeg\H{o}--Verblunsky theorem in the quaternionic setting (see \Cref{thm:QuaternionicSV}) and discuss the consequent characterisation of convergence of sequences of quaternions inside the unit ball.

    \item Once the relevant generalisations of the Wiener algebra have been introduced, we obtain via the matrix setting a quaternionic Baxter's theorem (see \Cref{cor:QuaternionicBaxter}): the Verblunsky coefficients of a non-trivial \(q\)-positive measure are summable if and only if the measure is absolutely continuous, and its Radon--Nikodym derivative is the restriction of an element of the quaternionic Wiener algebra.
\end{enumerate}

\section{The Schur Recurrences and Verblunsky's Formula In the Matrix Setting}
\label{sec:SchurRecVerblunskyForm}

We begin by abstracting to the setting of positive matrix-valued measures two very classical results: the classical Schur recurrences for the coefficients of the Schur iterates of a measure, and Verblunsky's formula for the moments of a measure. The primary novelty in this setting compared to the classical results is that the polynomials \(r_n\) and \(\tilde{r}_n\) are replaced by \emph{noncommutative rational functions}, as in noncommutative function theory (see e.g. \cite{PV21}). 

First, let us quickly recall a well-known construction, as in e.g. \cite{DPS08}. We say that a \(d\times d\) matrix-valued function \(f\) on \(\DD\) is a \emph{matrix-valued Herglotz function} if \(\Re F(z) = \frac12 (F(z) + F(z)^*)\) is positive semidefinite for all \(z \in \DD\), and we say that \(f\) is a  \emph{matrix-valued Schur function} if \(f(z)^*f(z) \leq I_d\) for all \(z \in \DD\). If a matrix-valued measure \(M\) has Verblunsky coefficients \((\a_n)_{n=0}^{\infty}\) then we also define for \(n = 0, 1, 2, \ldots\) the positive definite \emph{defect matrices} \(\r_n^L := (I_d - \a_n^*\a_n)^{\frac12}\) and \(\r_n^R := (I_d - \a_n\a_n^*)^{\frac12}\). Finally, we say that \(M\) is \emph{non-trivial} if the \(M\)-norms given by
\[
    \norm{f}_L = \left(\mathrm{Tr}\int_0^{2\pi} f(e^{i\th}) \, \rmd M(\th) f(e^{i\th})^*\right)^{\frac12} \quad \text{and} \quad \norm{f}_R = \left(\mathrm{Tr}\int_0^{2\pi} f(e^{i\th})^* \, \rmd M(\th) f(e^{i\th})\right)^{\frac12},
\]
for a matrix-valued polynomial \(f\), are positive definite. As \cite{DPS08} notes, in this case, \(\r_n^L\) and \(\r_n^R\) are invertible matrices for all \(n\).

Given a \(d \times d\) matrix-valued measure \(M\), we associate to \(M\) a Herglotz function and a Schur function as in the scalar setting: first, the Herglotz function associated to \(M\) is the holomorphic function \(\DD \to \CC^{d\times d}\) given by
\[
    F(z) = \int_0^{2\pi}\frac{e^{i\th} + z}{e^{i\th} - z} \, \rmd M(\th).
\]
The Schur function associated to \(M\) is the matrix Cayley transform of \(F\), i.e. the matrix-valued function \(f\) such that
\[
    F(z) = (I_d + zf(z))(I_d - zf(z))^{-1}.
\]
As discussed in \cite{DPS08}, these associations are bijective.

The matricial Verblunsky coefficients of \(M\) may then be obtained via the Schur algorithm applied to \(f\), as in \cite[Theorem 3.19]{DPS08}. Namely, defining \(f_0 := f\), for \(n = 0, 1, 2, \ldots\) the Verblunsky coefficients arise from the relations \(\a_n := f_n(0)\) and
\[
    f_{n+1}(z) = z^{-1} (\r_n^R)^{-1} (f_n(z) - \a_n)(I_d - \a_n^* f_n(z))^{-1} (\r_n^L)^{-1}.
\]
\begin{theorem}
\label{thm:MatrixSchurRec}
    Let \(M\) be a positive, non-trivial $d \times d$ matrix-valued measure on \(\TT\) with Verblunsky coefficients \((\a_n)_{n=0}^{\infty}\) and for \(n \in \NN\), let \(\r_n^L := (I_d - \a_n^* \a_n)^{\frac12}\) and \(\r_n^R := (I_d - \a_n \a_n^*)^{\frac12}\). Let \(f\) be the matrix-valued Schur function associated to \(M\), with coefficients \((s_k(f))_{k\in\NN}\). For \(k \in \NN\), these coefficients are given by
    \[
        s_k(f) = \r_0^R \ldots \r_{k-1}^R \a_k \r_{k-1}^L \ldots \r_0^L + R_k(\a_0, \ldots, \a_{k-1}, \a_0^*, \ldots, \a_{k-1}^*)
    \]
    where for each \(k\), \(R_k\) is a noncommutative rational function (see \cite{PV21}).
\end{theorem}
\begin{proof}
    Begin by recalling (3.58) of \cite{DPS08}:
    \[
        f_{n+1}(z) = z^{-1} \big(\r_n^R\big)^{-1}[f_n(z) - \a_n][I_d - \a_n^* f_n(z)]^{-1} \r_n^L,
    \]
    and rearrange this to see that
    \[
        f_n(z) - \a_n = z \r_n^R f_{n+1}(z)\big(\r_n^L\big)^{-1} - z \r_n^R f_{n+1}(z)\big(\r_n^L\big)^{-1} \a_n^* f_n(z).
    \]
    In terms of the coefficients of the Schur iterates, since \(f_n(0) = \a_n\), this becomes
    \begin{equation}
    \label{eq:FunctionRelationCoeffs}
        \sum_{k=1}^{\infty} s_k(f_n) z^k = \sum_{k=0}^{\infty} \r_n^R s_k(f_{n+1}) \big(\r_n^L\big)^{-1} z^{k+1} - \sum_{k=0}^{\infty} \left( \sum_{\ell = 0}^{k} \r_n^R s_{k - \ell}(f_{n+1})\big(\r_n^L\big)^{-1} \a_n^* s_{\ell}(f_n)\right) z^{k+1}.
    \end{equation}
    
    Notice first that the constant coefficient of both sides of \eqref{eq:FunctionRelationCoeffs} is zero. On the other hand, when \(k \geq 1\), by comparing coefficients of \(z^k\) in \eqref{eq:FunctionRelationCoeffs} we obtain the equality
    \begin{align*}
        s_k(f_n) & = \r_n^R s_{k-1}(f_{n+1})\big(\r_n^L\big)^{-1} - \sum_{\ell = 0}^{k-1} \r_n^R s_{k - \ell - 1}(f_{n+1}) \big(\r_n^L\big)^{-1} \a_n^* s_{\ell}(f_n) \\
        & = \r_n^R s_{k-1}(f_{n+1}) \big(\r_n^L \big)^{-1}[I - \a_n^* \a_n] - \sum_{\ell=1}^{k-1} \r_n^R s_{k - \ell - 1}(f_{n+1})\big(\r_n^L\big)^{-1} \a_n^* s_{\ell}(f_n);
    \end{align*}
    since \(\r_n^L = (I_d - \a_n^* \a_n)^{\frac12}\), we have for \(k \geq 1\) and \(n \in \NN\) that
    \begin{equation}
    \label{eq:CoeffsRelation}
        s_k(f_n) = \r_n^R s_{k-1}(f_{n+1}) \r_n^L - \sum_{\ell=1}^{k-1} \r_n^R s_{k - \ell - 1}(f_{n+1})\big(\r_n^L\big)^{-1} \a_n^* s_{\ell}(f_n).
    \end{equation}
    Compare this to (1.3.47) of \cite{Sim05a} in the classical setting, to which the above precisely reduces in the appropriate special case.

    Recursively apply \eqref{eq:CoeffsRelation} to the lower-degree terms \(s_{k - \ell - 1}(f_{n+1})\) and \(s_{\ell}(f_n)\), \(\ell = 1, \ldots, k - 1\). This results in the observation that for \(n \in \NN\) and \(k \geq 1\),
    \[
        R_{n,k} := - \sum_{\ell=1}^{k-1} \r_n^R s_{k - \ell - 1}(f_{n+1})\big(\r_n^L\big)^{-1} \a_n^* s_{\ell}(f_n)
    \]
    is a noncommutative rational function in the matrix variables \(\a_0, \ldots, \a_{k-1}, \a_0^*, \cdots, \a_{k-1}^*\), and additionally of \(\a_n, \a_n^*\) whenever \(n \geq k\). Notice in particular that the sum is empty when \(k = 1\), so that \(R_{n,1} \equiv 0\) for all \(n \in \NN\).

    It remains to use \eqref{eq:CoeffsRelation} recursively to see that, for \(k \geq 1\), we have
    \begin{align*}
        s_k(f) & = \r_0^R s_{k-1}(f_{1}) \r_0^L + R_{0,k}(\a_0, \ldots, \a_{k-1}, \a_0^*, \ldots, \a_{k-1}^*) \\
        & = \r_0^R \bigg( \r_1^R s_{k-2}(f_2) \r_1^R + R_{1,k-1} \bigg) \r_0^L + R_{0,k} \\
        & = \r_0^R \ldots \r_{k-1}^R s_0(f_k) \r_{k-1}^L \ldots \r_0^L + R_{0,k} + \sum_{\ell=0}^{k-2} \r_0^R \ldots \r_{\ell}^R R_{\ell + 1, k - \ell - 1} \r_{\ell}^L \ldots \r_0^L,
    \end{align*}
    allowing us to make some final observations. Firstly, since \(R_{n,1} \equiv 0\) for all \(n\), the \(\ell = k-2\) term of the summation is in fact zero. Moreover, notice that \(R_{0,k} + \sum_{\ell=0}^{k-3} \r_0^R \ldots \r_{\ell}^R R_{\ell + 1, k - \ell - 1} \r_{\ell}^L \ldots \r_0^L\) is a noncommutative rational expression involving noncommutative rational functions of the Verblunsky coefficients, and the highest-indexed Verblunsky coefficients appearing therein appear in \(R_{0,k}\). It follows that
    \[
        R_k := R_{0,k} + \sum_{\ell=0}^{k-2} \r_0^R \ldots \r_{\ell}^R R_{\ell + 1, k - \ell - 1} \r_{\ell}^L \ldots \r_0^L
    \]
    is a noncommutative rational function of the variables \(\a_0, \ldots, \a_{k-1}, \a_0^*, \ldots, \a_{k-1}^*\), and we are done.
\end{proof}

In the classical setting, one may use the Schur recurrences to prove Verblunsky's formula (see \cite{Sim05a}, Section 1.7, item 7 for a discussion). Our next result is a corresponding derivation in the matrix setting.

\begin{theorem}
\label{thm:MatrixVerForm}
    Let \(M\) be a positive $d \times d$ matrix-valued measure on \(\TT\) with moments \((C_n)_{n=-\infty}^{\infty}\) and Verblunsky coefficients \((\a_n)_{n=0}^{\infty}\). For \(n \in \NN\), let \(\r_n^L := (I_d - \a_n^* \a_n)^{\frac12}\) and \(\r_n^R := (I_d - \a_n \a_n^*)^{\frac12}\). Then for \(n \geq 1\), we have
    \[
         C_{n} = \r_0^R \cdots \r_{n-2}^R \a_{n-1} \r_{n-2}^L \cdots \r_0^L + \tilde{R}_{n}(\a_0, \ldots, \a_{n-2}, \a_0^*, \ldots, \a_{n-2}^*),
    \]
    where for each \(n \geq 1\), \(\tilde{R}_n\) is a noncommutative rational function.
\end{theorem}
\begin{proof}
    Let \(f\) be the $d\times d$ matrix-valued Schur function associated to \(M\), and let the corresponding Herglotz function be \(F\). Recall that \(F\) is given by
    \[
        F(z) = I_d + 2\sum_{n=1}^{\infty} C_n z^n,
    \]
    and moreover, that \(f\) and \(F\) are related via
    \[
        F(z) = \big( I_d + zf(z) \big) \big( I_d - zf(z) \big)^{-1}.
    \]
    Now, \(f\) is a Schur function, so for \(z \in \DD\), we have \(\norm{zf(z)} < 1\); it follows that \(\big(I_d - zf(z)\big)^{-1}\) is given by a geometric series and thus \(F\) is given by
    \[
        F(z) = \big( I_d + zf(z) \big) \sum_{n=0}^{\infty} (zf(z))^n = \sum_{n=0}^{\infty} (zf(z))^n + \sum_{n=0}^{\infty} (zf(z))^{n+1} = I_d + 2\sum_{n=1}^{\infty} (zf(z))^n.
    \]

    Equating our two expressions for \(F(z)\), we have that
    \[
        I_d + 2\sum_{n=1}^{\infty} C_n z^n = I_d + 2\sum_{n=1}^{\infty} (zf(z))^n,
    \]
    and it follows that
    \[
        \sum_{n=1}^{\infty} C_n z^n = \sum_{n=1}^{\infty} (zf(z))^n.
    \]
    The result will follow from comparing the coefficients of \(z^{n}\) for given \(n \geq 1\): on the left hand side, this is simply \(C_n\), so it remains to compute this coefficient in the expression \(\sum_{n=1}^{\infty} (zf(z))^n\), which we may do as follows. First, notice that for \(k > n\) we have \(s_n\big((zf(z))^k \big) = 0\), since \((zf(z))^k\) is a power series with lowest-order term \(z^k\), so we are in fact interested only in
    \[
        C_n = s_n\left(\sum_{k=1}^{\infty} (zf(z))^k \right) = s_n\left(\sum_{k=1}^{n} (zf(z))^k \right).
    \]
    Next, observe that
    \[
        s_n\big((zf(z))^k\big) = s_{n-k}(f^k),
    \]
    so that
    \[
        C_n = s_{n-1}(f) + \sum_{k=2}^n s_{n-k}(f^k).
    \]
    We may expand the latter sum as
    \[
        C_n = s_{n-1}(f) + \sum_{k=2}^n \sum_{\substack{i_1, \ldots, i_k \\ i_1 + \ldots + i_k = n-k}} s_{i_1}(f) \ldots s_{i_k}(f).
    \]

    To complete the calculation, notice that the second term is a noncommutative polynomial in lower-degree coefficients of \(f\), which by \Cref{thm:MatrixSchurRec} are noncommutative rational functions in lower-degree Verblunsky coefficients, while --- also by \Cref{thm:MatrixSchurRec} --- \(s_{n-1}(f)\) is given by
    \[
        s_{n-1}(f) = \r_0^R \cdots \r_{n-2}^R \a_{n-1} \r_{n-2}^L \cdots \r_0^L + R_{n-1}(\a_0, \ldots, \a_{n-2}, \a_0^*, \ldots, \a_{n-2}^*)
    \]
    where \(R_{n-1}\) is again a noncommutative rational function. Since noncommutative rational functions are closed under sums and products, we have that
    \begin{equation}
	\label{eq:VerForm}
        C_{n} = \r_0^R \cdots \r_{n-2}^R \a_{n-1} \r_{n-2}^L \cdots \r_0^L + \tilde{R}_{n}(\a_0, \ldots, \a_{n-2}, \a_0^*, \ldots, \a_{n-2}^*)
    \end{equation}
    where
    \begin{align*}
        \tilde{R}_n(\a_0, \ldots, \a_{n-2}, \a_0^*, \ldots, \a_{n-2}^*) & := R_{n-1}(\a_0, \ldots, \a_{n-2}, \a_0^*, \ldots, \a_{n-2}^*) \\ & \quad \quad \quad + \sum_{k=2}^n \sum_{\substack{i_1, \ldots, i_k \\ i_1 + \ldots + i_k = n-k}} s_{i_1}(f) \cdots s_{i_k}(f)
    \end{align*}
    is some noncommutative rational function of its arguments.
\end{proof}

\section{Quaternionic Preliminaries}

We shall first recall some notation and established ideas in quaternionic function theory. We liberally refer to Section 2 of \cite{ACKS15} in this section.

\subsection{Quaternions and Matrices of Quaternions} We denote by \(\HH\) the space of quaternions, whose elements are of the form \(p = a + bi + cj + dk\), where \(a,b,c,d \in \RR\) and \(i^2 = j^2 = k^2 = -1\); we may split such a \(p\) into its real part \(\Re(p) = a\) and imaginary part \(\Im(p) = bi + cj + dk\), and we write its conjugate as \(\overline{p} = a - bi - cj - dk\). The modulus or norm of a quaternion is given by \(\lvert p \rvert^2 = a^2 + b^2 + c^d + d^2\).

We say that \(p\) is \emph{purely imaginary} if \(a = 0\) and observe that if \(p\) is both purely imaginary and has norm 1, then \(p^2 = -1\), so that we have a full unit 2-sphere
\[
    \SS = \{bi + cj + dk : b,c,d \in \RR \text{ and } b^2 + c^2 + d^2 = 1\}
\]
of purely imaginary quaternions that square to \(-1\). Note that here \(\SS\) is distinct from the boundary of the unit ball in \(\HH\): writing
\[
    \BB = \{a + bi + cj + dk : a,b,c,d \in \RR \text{ and } a^2 + b^2 + c^2 + d^2 \leq 1\},
\]
we write \(\partial\BB\) for its boundary, i.e.
\[
    \partial\BB = \{a + bi + cj + dk : a,b,c,d \in \RR \text{ and } a^2 + b^2 + c^2 + d^2 = 1\}.
\]
The sphere \(\SS\) then consists of all purely imaginary elements of \(\partial\BB\).

It follows from the fact that any element of \(\SS\) squares to \(-1\) that the choice of basis \(i,j,k\) of \(\HH\) is arbitrary: if one chooses \(i' = bi + cj + dk, j' = yi + zj + wk \in \SS\), orthogonal in the sense that \(by + cz + dw = 0\), and sets \(k' =  i'j'\), then \(\{1, i', j', k'\}\) forms another basis for \(\HH\) with \((i')^2 = (j')^2 = (k')^2 = -1\).

Given an element \(i\) of \(\SS\), we write \(\CC_i\) for the complex plane generated by \(1\) and \(i\), 
\[
    \CC_i = \RR + \RR i = \{a + bi : a,b \in \RR\},
\]
and within \(\CC_i\) we have a copy of the unit circle \(\TT_i = \{a + bi : a,b \in \RR \text{ and } a^2 + b^2 = 1\}\). We shall often restrict a function on some subset of \(\HH\), \(f : X \subseteq \HH \to \HH\), to its domain inside \(\CC_i\), and we denote this restriction by \(f_i : X \cap \CC_i \to \HH\). Moreover, observe that if \(i,j \in \SS\) are orthogonal, then \(\{1, i, j, ij\}\) forms a basis for \(\HH\) as a vector space over \(\RR\). Thus any quaternion \(p\) may be written
\[
    p := a + bi + cj + dij = (a + bi) + (c + di)j = z + wj, \quad z, w \in \CC_i.
\]
This provides us with a natural embedding of the quaternions into \(\CC_i^{2\times 2}\) for any choice of orthogonal \(i,j \in \SS\): writing \(p = p^{(1)} + p^{(2)}j\) with \(p^{(1)}, p^{(2)} \in \CC_i\), we define
\[
    \x_i(p) := \begin{bmatrix} p^{(1)} & p^{(2)} \\[6pt] -\overline{p^{(2)}} & \overline{p^{(1)}}\end{bmatrix} \in \CC_i^{2\times 2},
\]
and it is straightforward to check that this definition of \(\x_i\) gives a unital, isometric \(*\)-homomorphism. 

We shall also have need of an extension of this idea to \emph{matrices} of quaternions: if \(A = [a_{k\ell}]_{k,\ell = 1}^{n} \in \HH^{n\times n}\), then choosing orthogonal \(i,j \in \SS\) write \(a_{k\ell} = a^{(1)}_{k\ell} + a^{(2)}_{k\ell}j\) where \(a^{(1)}_{k\ell}, a^{(2)}_{k\ell} \in \CC_i\) for all \(k, \ell = 1, \ldots, n\) and accordingly decompose \(A\) as \(A = A^{(1)} + A^{(2)}j\) where \(A^{(1)} := [a^{(1)}_{k\ell}]_{k,\ell = 1}^n\) and \(A^{(2)} := [a^{(2)}_{k\ell}]_{k,\ell=1}^n\) lie in \(\CC_i^{n\times n}\). The extension of \(\x_i\) to matrices is given by
\[
    \x_i(A) = \begin{bmatrix} A^{(1)} & A^{(2)} \\[6pt] -\overline{A^{(2)}} & \overline{A^{(1)}} \end{bmatrix} \in \CC_i^{2n \times 2n}.
\]
This idea goes at least as far back as the works of Brenner \cite{Bre51} and Lee \cite{Lee49} establishing the existence of right-eigenvalues for quaternionic matrices. 

\begin{definition}
    We say that a matrix \(A \in \HH^{n\times n}\) is \emph{positive semidefinite} if \(\xi^* A \xi \geq 0\) for all \(\xi \in \HH^n\). We say that \(A\) is \emph{positive definite} if \(\xi^* A \xi > 0\) for all nonzero \(\xi \in \HH^n\).
\end{definition}

It will also be useful for us to know the following fact, which was previously exploited in \cite{ACKS15}; we give a proof in the name of completeness.
\begin{lemma}
    \label{lem:UnitaryEquiv}
    Let \(i,j \in \SS\) be orthogonal. For each \(n\in\NN\), there exists some unitary \(U_n \in \CC_i^{2n \times 2n}\) with the property that, for any \(A = [a_{k\ell}]_{k,\ell=1}^n \in \HH^{n\times n}\), one has 
    \[
        \x_i(A) = U_n \begin{bmatrix} \x_i(a_{11}) & \cdots & \x_i(a_{1n}) \\ \vdots & \ddots & \vdots \\ \x_i(a_{n1}) & \cdots & \x_i(a_{nn}) \end{bmatrix} U_n^*.
    \]
\end{lemma}
\begin{proof}
    The unitary equivalence of the two matrices is in fact given by a permutation of the rows and columns, i.e. the matrix \(U_n\) is a permutation matrix. We claim that \(U_n = [u_{k\ell}]_{k,\ell=1}^{2n}\) is given by
    \[
        u_{k\ell} = \begin{cases}1, \text{ if } (k, \ell) \in \{(m, 2m-1), (n+m, 2m) : m = 1, \ldots, n\}, \\ 0, \text{ otherwise}.\end{cases}
    \]
    For example, for \(n = 2,3\) these are
    \[
        U_2 = \begin{bmatrix} 1 & 0 & 0 & 0 \\ 0 & 0 & 1 & 0 \\ 0 & 1 & 0 & 0 \\ 0 & 0 & 0 & 1 \end{bmatrix}, \quad U_3 = \begin{bmatrix} 1 & 0 & 0 & 0 & 0 & 0 \\ 0 & 0 & 1 & 0 & 0 & 0 \\ 0 & 0 & 0 & 0 & 1 & 0 \\ 0 & 1 & 0 & 0 & 0 & 0 \\ 0 & 0 & 0 & 1 & 0 & 0 \\ 0 & 0 & 0 & 0 & 0 & 1 \end{bmatrix}.
    \]
    Notice first that for any \(n \in \NN\), we may obtain \(U_{n}\) from the \(2n \times 2n\) matrix \(U_{n-1} \oplus I_2\) by moving the \((2n-1)^{\text{th}}\) row to position \(n+1\), and then shifting all following rows \(n+1, n+2, \ldots, 2n-2\) down one; in other words, we apply the permutation \(2n-1 \to n+1 \to n+2 \to \ldots \to 2n-2 \to 2n-1\). If we define the permutation matrix \(V_n\) for \(n \in \NN\) via
    \[
        V_{n+1} := I_{n} \oplus \begin{bmatrix} 0 & 1 & 0 & \cdots & 0 \\ \vdots & \ddots & \ddots & \ddots & \vdots \\ 0 & \cdots & 0 & 1 & 0 \\ 1 & 0 & \cdots & \cdots & 0 \\ 0 & \cdots & \cdots & 0 & 1 \end{bmatrix} \in \CC_i^{2(n+1) \times 2(n+1)},
    \]
    then this permutation may be described for \(n \in \NN\) by \(U_{n+1} = V_{n+1} \cdot (U_n \oplus I_2)\).
    
    Now, to see the claim, suppose for the purpose of induction that for some \(n \in \NN\) we had that
    \[
        \x_i(A) = U_n [\x_i(a_{k\ell})]_{k,\ell=1}^n U_n^*
    \]
    for all \(A = [a_{k\ell}]_{k,\ell=1}^n \in \HH^{n\times n}\). Let \(A = [a_{k\ell}]_{k,\ell=1}^{n+1} \in \HH^{(n+1)\times(n+1)}\) and calculate for an arbitrary block matrix \(\begin{bmatrix} X & Y \\ Z & W \end{bmatrix}\) that for a appropriately-sized matrices \(\Gamma\) and \(I_k\),
    \[
        \begin{bmatrix} \Gamma & 0 \\ 0 & I_k \end{bmatrix} \begin{bmatrix} X & Y \\ Z & W \end{bmatrix} \begin{bmatrix} \Gamma & 0 \\ 0 & I_k\end{bmatrix}^* = \begin{bmatrix} \Gamma X \Gamma^* & \Gamma Y \\ Z \Gamma^* & W \end{bmatrix}.
    \]
    By splitting \([\x_i(a_{k,\ell})]_{k,\ell=1}^{n+1}\) into the appropriate blocks, then, we have that
    \begin{equation}
    \label{eq:UnitaryEquivLemma}
        \begin{split}
            U_{n+1} [\x_i(a_{k,\ell})]_{k,\ell=1}^{n+1} U_{n+1}^* & = V_{n+1} \begin{bmatrix} U_n & 0 \\ 0 & I_2 \end{bmatrix} [\x_i(a_{k,\ell})]_{k,\ell=1}^{n+1} \begin{bmatrix} U_n & 0 \\ 0 & I_2 \end{bmatrix}^* V_{n+1}^* \\
            & = V_{n+1} \begin{bmatrix} U_n [\x_i(a_{k,\ell})]_{k,\ell=1}^n U_n^* & U_n [\x_i(a_{k,n+1})]_{k=1}^n \\[6pt] [\x_i(a_{n+1,\ell})]_{\ell=1}^n U_n^* & \x_i(a_{n+1}) \end{bmatrix} V_{n+1}^*. 
        \end{split}
    \end{equation}
    By the inductive hypothesis, the \((1,1)\)-block of the centre matrix is simply \(\x_i(A')\). Expand the right-most column and bottom row to write the right-hand side as
    \[
        V_{n+1} \left[\begin{array}{c|c} \x_i(A') & U_n \begin{bmatrix}a_{1,n+1}^{(1)} & a_{1,n+1}^{(2)} \\[6pt] -\overline{a_{1,n+1}^{(2)}} & \overline{a_{1,n+1}^{(1)}} \\ \vdots & \vdots \\ a_{n,n+1}^{(1)} & a_{n,n+1}^{(2)} \\[6pt] -\overline{a_{n,n+1}^{(2)}} & \overline{a_{n,n+1}^{(1)}} \end{bmatrix} \\[50pt] \hline \\[-8pt] \begin{bmatrix}a_{n+1,1}^{(1)} & a_{n+1,1}^{(2)} & \cdots & a_{n+1,n}^{(1)} & a_{n+1,n}^{(2)} \\[6pt] -\overline{a_{n+1,1}^{(2)}} & \overline{a_{n+1,1}^{(1)}} & \cdots & -\overline{a_{n+1,n}^{(2)}} & \overline{a_{n+1,n}^{(1)}}\end{bmatrix} U_n^* & \begin{array}{cc} a_{n+1,n+1}^{(1)} & a_{n+1,n+1}^{(2)} \\[6pt] -\overline{a_{n+1,n+1}^{(2)}} & \overline{a_{n+1, n+1}^{(1)}}\end{array} \end{array}\right] V_{n+1}^*.
    \]
    Of course, multiplication by \(U_n\) on the left and by \(U_n^*\) on the right simply permutes the rows and columns, respectively, so applying these permutations we obtain the matrix
    \[
        V_{n+1} \left[\begin{array}{c|c} \x_i(A') & \begin{array}{cc} a_{1,n+1}^{(1)} & a_{1,n+1}^{(2)} \\[6pt] \vdots & \vdots \\ a_{n,n+1}^{(1)} & a_{n,n+1}^{(2)} \\[6pt] -\overline{a_{1,n+1}^{(2)}} & \overline{a_{1,n+1}^{(1)}} \\ \vdots & \vdots \\ -\overline{a_{n,n+1}^{(2)}} & \overline{a_{n,n+1}^{(1)}} \end{array} \\[10pt] \\[-8pt] \hline \\[-8pt] \begin{array}{cccccc} a_{n+1,1}^{(1)} & \cdots &  a_{n+1,n}^{(1)} & a_{n+1,1}^{(2)} & \cdots & a_{n+1,n}^{(2)} \\[6pt] -\overline{a_{n+1,1}^{(2)}} & \cdots & -\overline{a_{n+1,n}^{(2)}} & \overline{a_{n+1,1}^{(1)}} & \cdots & \overline{a_{n+1,n}^{(1)}}\end{array} & \begin{array}{cc} a_{n+1,n+1}^{(1)} & a_{n+1,n+1}^{(2)} \\[6pt] -\overline{a_{n+1,n+1}^{(2)}} & \overline{a_{n+1, n+1}^{(1)}}\end{array} \end{array}\right] V_{n+1}^*.
    \]

    Now, we have described above the permutation on rows arising from multiplying by the matrix \(V_{n+1}\) on the left; multiplying by \(V_{n+1}^*\) on the right performs the same permutation on the columns. Thus --- applying first the permutation on the rows --- our matrix becomes
    \[
        \begin{bmatrix} a_{11}^{(1)} & \cdots & a_{1n}^{(1)} & a_{11}^{(2)} & \cdots & a_{1n}^{(2)} & a_{1,n+1}^{(1)} & a_{1,n+1}^{(2)} \\[6pt] \vdots & \ddots & \vdots & \vdots & \ddots & \vdots & \vdots & \vdots \\[6pt] a_{n1}^{(1)} & \cdots & a_{nn}^{(1)} & a_{n1}^{(2)} & \cdots & a_{nn}^{(2)}  & a_{n,n+1}^{(1)} & a_{n,n+1}^{(2)} \\[6pt] a_{n+1,1}^{(1)} & \cdots & a_{n+1,n}^{(1)} & a_{n+1,1}^{(2)} & \cdots & a_{n+1,n}^{(2)}& a_{n+1,n+1}^{(1)} & a_{n+1,n+1}^{(2)} \\[6pt] -\overline{a_{11}^{(2)}} & \cdots & -\overline{a_{1n}^{(2)}} & \overline{a_{11}^{(1)}} & \cdots & \overline{a_{1n}^{(1)}} & -\overline{a_{1,n+1}^{(2)}} & \overline{a_{1,n+1}^{(1)}} \\[6pt] \vdots & \ddots & \vdots & \vdots & \ddots & \vdots & \vdots & \vdots \\[6pt] -\overline{a_{n1}^{(2)}} & \cdots & -\overline{a_{nn}^{(2)}} & \overline{a_{n1}^{(1)}} & \cdots & \overline{a_{nn}^{(1)}} & -\overline{a_{n,n+1}^{(2)}} & \overline{a_{n,n+1}^{(1)}} \\[6pt] -\overline{a_{n+1,1}^{(2)}} & \cdots & -\overline{a_{n+1,n}^{(2)}} & \overline{a_{n+1,1}^{(1)}} & \cdots & \overline{a_{n+1,n}^{(1)}} & -\overline{a_{n+1,n+1}^{(2)}} & \overline{a_{n+1,n+1}^{(1)}} \end{bmatrix} V_{n+1}^*,
    \]
    and finally, by permuting the columns in the same manner we arrive at
    \[
        \begin{bmatrix} a_{11}^{(1)} & \cdots & a_{1n}^{(1)} & a_{1,n+1}^{(1)} &  a_{11}^{(2)} & \cdots & a_{1n}^{(2)} & a_{1,n+1}^{(2)} \\[6pt] \vdots & \ddots & \vdots & \vdots & \vdots & \ddots & \vdots & \vdots \\[6pt] a_{n1}^{(1)} & \cdots & a_{nn}^{(1)} & a_{n,n+1}^{(1)} & a_{n1}^{(2)} & \cdots & a_{nn}^{(2)} & a_{n,n+1}^{(2)} \\[6pt] a_{n+1,1}^{(1)} & \cdots & a_{n+1,n}^{(1)} & a_{n+1,n+1}^{(1)} & a_{n+1,1}^{(2)} & \cdots & a_{n+1,n}^{(2)} & a_{n+1,n+1}^{(2)} \\[6pt] -\overline{a_{11}^{(2)}} & \cdots & -\overline{a_{1n}^{(2)}} & -\overline{a_{1,n+1}^{(2)}} & \overline{a_{11}^{(1)}} & \cdots & \overline{a_{1n}^{(1)}} & \overline{a_{1,n+1}^{(1)}} \\[6pt] \vdots & \ddots & \vdots & \vdots & \vdots & \ddots & \vdots & \vdots \\[6pt] -\overline{a_{n1}^{(2)}} & \cdots & -\overline{a_{nn}^{(2)}} & -\overline{a_{n,n+1}^{(2)}} & \overline{a_{n1}^{(1)}} & \cdots & \overline{a_{nn}^{(1)}} & \overline{a_{n,n+1}^{(1)}} \\[6pt] -\overline{a_{n+1,1}^{(2)}} & \cdots & -\overline{a_{n+1,n}^{(2)}} & -\overline{a_{n+1,n+1}^{(2)}} & \overline{a_{n+1,1}^{(1)}} & \cdots & \overline{a_{n+1,n}^{(1)}} & \overline{a_{n+1,n+1}^{(1)}} \end{bmatrix} = \x_i(A).
    \]
    Hence the right-hand side of \eqref{eq:UnitaryEquivLemma} may now be readily seen to be \(\x_i(A)\).

    It remains to note the base case: when \(n=1\), \(A = a_{11} \in \HH\) and furthermore \(U_1 = I_2\), so
    \[
        U_1 \x_i(a_{11}) U_1^* = \x_i(a_{11}) = \x_i(A).
    \]
    By induction, then, \(U_n [\x_i(a_{k\ell})]_{k,\ell=1}^n U_n^* = \x_i(A)\) for all \(n \in \NN\) and \(A = [a_{k\ell}]_{k,\ell=1}^n \in \HH^{n\times n}\).
\end{proof}

\begin{remark}
	This lemma is of particular interest to us because this extension of \(\x_i\) to matrices preserves positivity of the matrix, but does not preserve certain other important properties --- notably, if \(A\) is Toeplitz then \(\x_i(A)\) is in general not (block-)Toeplitz. Nevertheless, \Cref{lem:UnitaryEquiv} shows that \(\x_i(A)\) is \emph{unitarily equivalent to} an appropriate block-Toeplitz matrix.
\end{remark}

\subsection{Slice Hyperholomorphy} Slice hyperholomorphy of functions \(U \subseteq \HH \to \HH\) may be defined via the Cauchy-Riemann equations, analogous to the usual setting of complex analysis but with the added subtlety that an \(\HH\)-valued \(f\) will not in general commute with \(i\). To wit: 
\begin{definition}
	Let \(U \subseteq \HH\) be open and \(f : U \to \HH\) be real differentiable as a vector function (considering \(\HH\) as \(\RR^4\)). We say that \(f\) is \emph{left-slice hyperholomorphic} if, for all \(i \in \SS\),
	\[
	    \frac{\partial}{\partial x}f_i(x + yi) + i\frac{\partial}{\partial y}f_i(x + yi) = 0,
	\]
	and we say that \(f\) is \emph{right-slice hyperholomorphic} if, for all \(i \in \SS\),
	\[
	    \frac{\partial}{\partial x}f_{i}(x + yi) + \frac{\partial}{\partial y} f_{i}(x + yi) i = 0.
	\]
\end{definition}

We recall that monomials with coefficients on the right, i.e. of the form \(p \mapsto p^n \a\) for \(\a \in \HH\), are left-slice hyperholomorphic, while monomials with coefficients on the left, i.e. \(p \mapsto \a p^n\), \(\a \in \HH\), are right-slice hyperholomorphic, and it follows that the same is true of finite sums of such monomials, and indeed also of power series centred at the origin (when convergent; see \cite{CSS11}). Furthermore, as \cite{ACKS15} remarks, this is in fact also true of power series centred at nonzero reals, but not of power series centred off of the real line.

We shall denote respectively by \(\HH[p]^L\) and \(\HH[p]^R\) the spaces of, respectively, left- and right-slice hyperholomorphic polynomials, that is, we define
\[
    \HH[p]^L := \bigcup_{n=0}^{\infty}\left\{p \mapsto \sum_{k=0}^{n} p^k \p_k : n \in \NN, \p_0, \ldots, \p_n \in \HH\right\}
\]
and
\[
    \HH[p]^R := \bigcup_{n=0}^{\infty}\left\{p \mapsto \sum_{k=0}^{n} \p_k p^k : n \in \NN, \p_0, \ldots, \p_n \in \HH\right\}.
\]

The pointwise product of two (left- or right-) slice hyperholomorphic functions is not, in general, slice hyperholomorphic. In place of the pointwise product, the \(\star\)-product of two slice hyperholomorphic functions, given by convolution of their coefficients, is often considered; see e.g. \cite{ACS16a}.

\begin{definition}
	Let \(m, n \in \NN\).

	If \(\p(p) = \p_0 + p\p_1 + \ldots + p^n \p_n, \q(p) = \q_0 + p \q_1 + \ldots + p^m \q_m \in \HH[p]^L\), then the \emph{\(\star\)-product} of \(\p\) and \(\q\) is given by
	\[
		(\p \star \q)(p) = \sum_{\ell = 0}^{n+m} p^{\ell} c_\ell,
	\]
	where
	\[
		c_{\ell} = \sum_{\a + \b = \ell} \p_\a \q_\b.
	\]

	Similarly, if \(\p(p) = \p_0 + \p_1 p + \ldots + \p_n p^n, \q(p) = \q_0 + \q_1 p + \ldots + \q_m p^m \in \HH[p]^R\), their \emph{\(\star\)-product} is given by
	\[
		(\p \star \q)(p) = \sum_{\ell = 0}^{n+m} c_\ell p^{\ell},
	\]
	where \(c_\ell\) is as above.
\end{definition}

We define the following natural mappings from \(\HH[p]^L\) and \(\HH[p]^R\) into the space \(\CC_i^{2\times2}[z]\) of analytic polynomials with coefficients in \(\CC_i^{2\times2}\).
\begin{definition}
	The mappings \(\P_i^L : \HH[p]^L \to \CC_i^{2\times2}[z]\) and \(\P_i^R : \HH[p]^R \to \CC_i^{2\times2}[z]\) are given by
	\[
	    \P_i^L\left(\sum_{k=0}^{n} p^k \p_k\right) := \sum_{k=0}^n \x_i(\p_k) z^k =: \P_i^R\left(\sum_{k=0}^{n} \p_k p^k\right),
	\]
    i.e. we apply \(\x_i\) coefficient-wise and not to the quaternionic variable \(p\). The inverses of \(\P_i^L\) and \(\P_i^R\) are defined on the subspace of \(\CC_i^{2\times 2}[z]\) consisting of matrix-valued analytic polynomials whose coefficients lie in the image of \(\x_i\).
\end{definition}

\subsection{\(q\)-positive Measures}

In one perspective on the classical theory of orthogonal polynomials on the unit circle, one begins with a positive measure on the unit circle \(\TT\) in \(\CC\). A known analogue in the quaternions of such a measure is that of the \emph{\(q\)-positive} measure, introduced in \cite{ACKS15} and studied further in \cite{ABCKS16} (note that while we only herein require the \(\HH\)-valued case, the general \(\HH^{n\times n}\)-valued case has been previously studied):

\begin{definition}
    \label[definition]{def:qPosMeasure}
    Given orthogonal \(i,j \in \SS\), let \(\mu\) be a quaternion-valued measure on \(\TT_i\) and write \(\mu = \mu^{(1)} + \mu^{(2)}j\) where \(\mu^{(1)}, \mu^{(2)}\) are \(\CC_i\)-valued measures. We say that \(\mu\) is \emph{\(q\)-positive} if the \(\CC_i^{2\times 2}\)-valued measure
    \[
        M_i := \begin{bmatrix} \mu^{(1)} & \mu^{(2)} \\[4pt] \overline{\mu^{(2)}} & \nu \end{bmatrix}
    \]
    is positive (that is, positive semidefinite-valued), where
    \[
        \int_0^{2\pi} e^{in\th} \, \rmd \nu := \int_0^{2\pi} e^{-in\th} \, \rmd \mu^{(1)}(\th) \quad \text{for } n \in \mathbb{Z},
    \]
    \(\mu^{(1)}\) (and hence \(\nu\)) is a positive measure, and $\mu^{(2)}$ obeys
    \[
        \int_0^{2\pi} e^{in\th} \, \rmd \mu^{(2)}(\th) = - \int_0^{2\pi} e^{-in\th} \, \rmd\mu^{(2)}(\th), \quad \text{for } n \in \mathbb{Z}.
    \]
\end{definition}

Notice that a \(q\)-positive measure is a \emph{local} object, relying upon a choice of \(i,j\in\SS\); nevertheless, there exists a corresponding \emph{global} object as with the classical theory, namely, the sequence of moments of \(\mu_i\). Intimately related, both classically and --- as we shall see --- in this quaternionic theory, is the Toeplitz matrix of moments of a measure.

Now we shall introduce the notion of a positive definite sequence in the quaternionic setting, and also a related concept of ``non-triviality" which directly generalises the idea from the usual theory of orthogonal polynomials on the unit circle (\cite{Sim05a}); we shall assume non-triviality of our positive definite sequences throughout the remainder of the paper.

\begin{definition}
    Let \(c = (c_n)_{n=-\infty}^{\infty}\) be a bi-infinite sequence in \(\HH\). We say that \(c\) is \emph{positive definite} if for each \(n \in \NN\) the Toeplitz matrix
	\[
	    T_n(c) := [c_{j-k}]_{k,j=0}^n = \begin{bmatrix}
	        c_0 & c_{1} & \cdots & c_{n} \\ 
	        c_{-1} & c_0 & \cdots & c_{n-1} \\
	        \vdots & \vdots & \ddots & \vdots \\
	        c_{-n} & c_{-(n-1)} & \cdots & c_0
	    \end{bmatrix}
	\]
	is positive semidefinite.

    If in fact \(T_n(c)\) is positive definite for all \(n \in \NN\), we say that \(c\) is a \emph{non-trivial} positive sequence.
\end{definition}

\begin{remark}
    The connection between this notion of non-triviality, the notion used for a matrix-valued measure in Section 2, and a positive measure on the unit circle not being finitely atomic in the classical setting, can be seen from Lemma 2.1 of \cite{DPS08} and the subsequent remark, and the related discussion of \cite{Sim05a}.
\end{remark}

We remark that a particular consequence of positive definiteness of \(c\) is that \(c_{-n} = \overline{c_n}\).

\begin{definition}
    Let \(i,j \in \SS\) be orthogonal and let \(\mu_i\) be a \(q\)-positive measure on \(\TT_i\). The \emph{moments} of \(\mu_i\) are
    \[
        c_n := \int_0^{2\pi} e^{i n \th} \, \rmd \mu_i(\th), \quad n \in \ZZ.
    \]
    We say that \(\mu_i\) is \emph{non-trivial} if its sequence of moments \((c_n)_{n=-\infty}^{\infty}\) is a non-trivial sequence, and we say that \(\mu_i\) is a \emph{probability \(q\)-positive measure} if \(\int_0^{2\pi} \, \rmd\mu_i(\th) = 1\), i.e. if \(c_0 = 1\).
\end{definition}

Definitionally, one may always obtain a positive definite sequence from a positive measure. The quaternionic Herglotz theorem of \cite{ACKS15} assures us that in fact, for any positive definite sequence \(c = (c_n)_{n=-\infty}^{\infty} \subseteq \HH\) and orthogonal, purely imaginary \(i,j \in \SS\) there exists a \emph{unique} \(q\)-positive measure \(\mu_{i}\) on \(\TT_{i}\) with moment sequence \(c\):

\begin{theorem}[\cite{ACKS15}, Theorem 5.5]
    The sequence \((c_n)_{n=-\infty}^{\infty}\) is positive definite if and only if, for some (and hence all) orthogonal \(i,j\in\SS\), there exists a unique \(q\)-positive measure \(\mu_i\) on \(\TT_i\) such that
    \[
        c_n = \int_0^{2\pi} e^{i n \th} \, \rmd\mu_i(\th), \quad n \in \ZZ.
    \]
\end{theorem}

This bijective relationship allows us to abstract away from the classical setting of measures on \(\TT\) to consider the global object \(c\), a doubly-infinite positive definite sequence of quaternions.

\subsection{Consistency of Verblunsky Coefficients}
\label{subsec:ConsistencyVCs}

In this subsection, we discuss the two methods by which one might acquire quaternionic Verblunsky coefficients and see that they in fact coincide.

\begin{definition}
    Let \(M\) be a matrix-valued measure on \(\TT\). By the \emph{Verblunsky coefficients of \(M\)}, we mean the unique sequence of complex matrices \((\a_n)_{n=0}^{\infty}\) associated to \(M\) by \cite[Theorem 3.12]{DPS08}.
\end{definition}

\begin{lemma}
\label{lem:ImageOfXi}
    Let \(i,j \in \SS\) be orthogonal and let \(M_i\) be a \(\CC_i^{2\times2}\)-valued probability measure on \(\TT_i\). The moments \((C_n)_{n=-\infty}^{\infty}\) lie in the image of \(\x_i\) if and only if the Verblunsky coefficients \((\a_n)_{n=-\infty}^{\infty}\) do.
\end{lemma}
\begin{proof}
    Since \(\x_i\) is in particular a \(*\)-homomorphism, we see from \Cref{thm:MatrixVerForm} that if \((\a_n)_{n=0}^{\infty}\) lies entirely in the image of \(\x_i\), so too must \((C_n)_{n = 1}^{\infty}\) and hence \((C_n)_{n\in\ZZ\setminus\{0\}}\). On the other hand, since \(M\) is a probability measure, \(C_0 = I = \x_i(1)\), so that \((C_n)_{n=-\infty}^{\infty}\) is contained wholly in the image of \(\x_i\).

    Conversely, we can rearrange the conclusion of \Cref{thm:MatrixVerForm} to
    \[
        \a_{n-1} = (\r_{n-2}^R)^{-1} \cdots (\r_0^R)^{-1} \left(C_n - \tilde{R}_n(\a_0, \ldots, \a_{n-2}, \a_0^*, \ldots, \a_{n-2}^*)\right) (\r_0^L)^{-1} \cdots (\r_{n-2}^L)^{-1}
    \]
    and then it follows inductively that if \((C_n)_{n=-\infty}^{\infty}\) lies entirely in the image of \(\x_i\), so does \((\a_n)_{n=0}^{\infty}\).
\end{proof}

Let \(c = (c_n)_{n=-\infty}^{\infty}\) be a non-trivial positive definite sequence with \(c_0 = 1\), and for any \(i' \in \SS\), let \(\mu_{i'}\) be the unique \(q\)-positive measure on \(\TT_{i'}\) with moments \(c\). As a consequence of the quaternionic Herglotz representation theorem of \cite{ACKS15}, the functions \(F_{i'} : \DD_{i'} \to \HH\) given by
\[
    F_{i'}(z) := \int_{0}^{2\pi} \frac{1 + ze^{-i'\th}}{1 - ze^{-i'\th}} \, \rmd\mu_{i'}(\th)
\]
satisfies \(\Re F_{i'}(z) \geq 0\). A calculation almost identical to that of the classical setting allows one to express \(F_{i'}\) as
\[
    F_{i'}(z) = 1 + 2 \sum_{n=1}^{\infty} z^k c_k.
\]
Make the global definition, for \(p \in \BB\), of
\[
    F(p) := 1 + 2 \sum_{n=1}^{\infty} p^k c_k.
\]
Certainly there is a unique \(F\) for each \(c\). Moreover, such an \(F\) is automatically left-slice hyperholomorphic by virtue of its power series representation, and since any \(p \in \BB\) lies in \(\DD_{i'}\) for some \(i' \in \SS\), we have \(\Re F(p) = \Re F_{i'}(p) \geq 0\). 

Next we Cayley transform \(F\) to obtain a quaternionic Schur function \(f\), i.e. define
\[
    f(p) := (1 + F(p)) \star (1 - F(p))^{-\star};
\]
again, this association is bijective. 

\begin{definition}
    Let \(c = (c_n)_{n=-\infty}^{\infty}\) be a non-trivial positive definite sequence with \(c_0 = 1\), and let \(f\) be the quaternionic Schur function associated to \(c\) as above. By the \emph{Verblunsky coefficients of} \(c\), we mean the unique sequence of quaternions \((\g_n)_{n=0}^{\infty}\) associated to \(f\) by the quaternionic Schur algorithm (see \cite{ACS16b} as well as the discussion following Remark 10.1.2 of \cite{ACS16a} where the \(\g_n\) are called \emph{Schur coefficients}).
\end{definition}

\begin{remark}
    The quaternionic Schur algorithm first appeared in the paper \cite{ACS16b}; the book \cite{ACS16a} provides authoritative further exposition on both the Schur algorithm and, more broadly, Schur analysis in the quaternionic setting.
\end{remark}

Let \((\g_n)_{n=0}^{\infty}\) be the Verblunsky coefficients of \(c\). For \(n\in\NN\), define \(\a_n := \x_i(\g_n) \in \CC_i^{2\times 2}\) and use the matrix Verblunsky's theorem \cite[Theorem 3.12]{DPS08} to obtain \(M_i\), the unique positive matrix-valued measure with Verblunsky coefficients \((\a_n)_{n=0}^{\infty}\). Further denote the moments of \(M_i\) by \((C_n)_{n=-\infty}^{\infty} \subseteq \CC_i^{2\times2}\).

On the other hand, if for \(n\in\ZZ\) we define \(\tilde{C}_n := \x_i(c_n) \in \CC_i^{2\times 2}\), then the solution to the matrix-valued moment problem gives a unique positive matrix-valued measure with moment sequence \((\tilde{C}_n)_{n=-\infty}^{\infty}\), which we denote \(\tilde{M}_i\); let the Verblunsky coefficients of this measure be \((\tilde{\a}_n)_{n=0}^{\infty}\). Observe that since \((\tilde{C}_n)_{n=-\infty}^{\infty}\) lies in the image of \(\x_i\), so too does \((\tilde{\a}_n)_{n=0}^{\infty}\), so \(\tilde{\a}_n \tilde{\a}_n^* = \tilde{\a}_n^* \tilde{\a}_n\) and hence we may define \(\tilde{\r}_n := (I - \tilde{\a}_n \tilde{\a}_n^*)^{\frac12} = (I - \tilde{\a}_n^* \tilde{\a}_n)^{\frac12}\). It follows from \Cref{thm:MatrixVerForm} that
\[
    C_{n} = \r_0 \cdots \r_{n-2} \a_{n-1} \r_{n-2} \cdots \r_0 + \tilde{R}_{n}(\a_0, \ldots, \a_{n-2}, \a_0^*, \ldots, \a_{n-2}^*).
\]
and
\[
    \tilde{C}_n = \tilde{\r}_0 \cdots \tilde{\r}_{n-2} \tilde{\a}_{n-1} \tilde{\r}_{n-2} \cdots \tilde{\r}_0 + \tilde{R}_{n}(\tilde{\a}_0, \ldots, \tilde{\a}_{n-2}, \tilde{\a}_0^*, \ldots, \tilde{\a}_{n-2}^*)
\]
Now, the right-hand sides both of these relations are rational expressions in, respectively, \((\a_n)_{n=0}^{\infty}\), \((\a_n^*)_{n=0}^{\infty}\) and \((\tilde{\a}_n)_{n=0}^{\infty}\), \((\tilde{\a}_n^*)_{n=0}^{\infty}\), which are by \Cref{lem:ImageOfXi} sequences which lie entirely in the image of \(\x_i\). Let \((\tilde{\g}_n)_{n=-\infty}^{\infty}\) be the preimages of \((\tilde{\a}_n)_{n=0}^{\infty}\) under \(\x_i\). Taking \(\x_i^{-1}\) of the formula for \(\tilde{C}_n\) then gives an essentially identical formula for \(c_n\) in terms of \((\tilde{\g}_n)_n\), i.e. \((\tilde{\g}_n)_{n=0}^{\infty}\) is another sequence of quaternionic Verblunsky coefficients associated to \(c\).

We have now bijectively associated two sequences of quaternionic Verblunsky coefficients to \(c\) --- \((\g_n)_{n=0}^{\infty}\) and \((\tilde{\g}_n)_{n=0}^{\infty}\) --- and so, by the uniqueness of quaternionic Verblunsky coefficients \cite{ACS16a}, these sequences must coincide.

\section{Quaternionic Orthogonal Polynomials}

Given orthogonal \(i,j \in \SS\) and a non-trivial \(q\)-positive probability measure \(\mu_i\) on \(\TT_i\) with moments \(c \subseteq \HH\), the immediate left and right generalisations of an inner product (on polynomials) to this setting,
\[
    \int_0^{2\pi} \overline{\q(e^{i\th})} \, \rmd \mu_i(\th) \p(e^{i\th}) \quad \text{ and } \quad \int_0^{2\pi} \p(e^{i\th}) \, \rmd \mu_i(\th) \overline{\q(e^{i\th})},
\]
fail in general to be definite forms (due to \(\mu_i\) being quaternion-valued). We give an alternate definition of an inner product with respect to a \(q\)-positive measure and obtain an integral form for it in terms of the decomposition \(\mu_i = \mu_i^{(1)} + \mu_i^{(2)}j\), for polynomials in both \(\HH[p]^L\) and \(\HH[p]^R\).

\begin{definition}
\label{def:R-IP}
	Let
	\[
	    \p(p) = \sum_{k=0}^n p^k \p_k, \quad \q(p) = \sum_{\ell = 0}^m p^k \q_k \in \HH[p]^L;
	\]
	if \(m < n\), set \(\q_k = 0\) for \(m < k \leq n\) and symmetrically if \(m > n\) set \(\p_k = 0\) for \(n < k \leq m\), and define their coefficient vectors
	\[
	    \hat{\p} := \begin{bmatrix} \p_0 \\ \vdots \\ \p_N\end{bmatrix}, \quad \hat{\q} := \begin{bmatrix} \q_0 \\ \vdots \\ \q_N\end{bmatrix}
	\]
	where \(N = \max\{n,m\}\). 
	We define
	\[
	    \langle \p, \q \rangle_R := \hat{\q}^* T_N(c) \hat{\p}.
	\]
\end{definition}

\begin{remark}
	By non-triviality of \(\mu_i\) and hence of \(c\), this indeed defines an inner product --- i.e. a positive definite sesquilinear form --- on \(\HH[p]^L\). Note that the reason we call this form \(\langle\cdot,\cdot\rangle_R\) is because we shall see later that the orthogonal polynomials in \(\HH[p]^L\) are closely related to the right-orthogonal polynomials of the matrix setting.
\end{remark}

\begin{theorem}
\label{thm:R-IP}
	Let \(\p, \q \in \HH[p]^L\). Then we may rewrite \( \langle \p, \q \rangle_R\) as
	\[
		\langle \p, \q \rangle_R = \int_0^{2\pi} \overline{\q(e^{i\th})} \, \rmd\mu_i^{(1)}(\th) \p(e^{i\th}) + \int_0^{2\pi} \overline{\q(e^{i\th})} \, \rmd\mu_i^{(2)}(\th) j \p(e^{-i\th}),
	\]
	and moreover we may express this as a perturbation of the ``expected" sesquilinear form as
	\[
		\langle \p, \q \rangle_R = \int_0^{2\pi} \overline{\q(e^{i\th})} \, \rmd\mu_i(\th) \p(e^{i\th}) + \int_0^{2\pi} \overline{\q(e^{i\th})} \, \rmd \mu_i^{(2)}(\th) j \left(\p(e^{-i\th}) - \p(e^{i\th}) \right).
	\]
\end{theorem}
\begin{proof}
	Calculate directly that
	\begin{align*}
	    \langle \p, \q \rangle_R & = \sum_{\ell = 0}^N \sum_{k = 0}^N \overline{\q_\ell} c_{k-\ell} \p_k = \sum_{\ell = 0}^m \sum_{k = 0}^n \overline{\q_\ell} c_{k-\ell} \p_k  \\
	    & = \sum_{\ell = 0}^m \sum_{k = 0}^n \overline{\q_\ell} \left(\int_0^{2\pi} e^{i(k-\ell)\th} \, \rmd \mu_i(\th) \right) \p_k \\
	    & = \sum_{\ell=0}^m \sum_{k=0}^n \int_0^{2\pi} \overline{\q_\ell} e^{-i\ell\th} e^{ik\th} \, \rmd \mu_i(\th) \p_k \\
	    & = \sum_{k=0}^n \int_0^{2\pi} \overline{\q(e^{i\th})} e^{ik\th} \, \rmd \mu_i(\th) \p_k \\
	    & = \sum_{k=0}^n \int_0^{2\pi} \overline{\q(e^{i\th})} \, \rmd\mu_i^{(1)}(\th) e^{ik\th} \p_k + \int_0^{2\pi} \overline{\q(e^{i\th})} \, \rmd\mu_i^{(2)}(\th)j e^{-ik\th} \p_k \\
	    & = \int_0^{2\pi} \overline{\q(e^{i\th})} \, \rmd\mu_i^{(1)}(\th) \p(e^{i\th}) + \int_0^{2\pi} \overline{\q(e^{i\th})} \, \rmd\mu_i^{(2)}(\th) j \p(e^{-i\th})
	\end{align*}
	(since \(\mu_i^{(1)}\) is real valued, \(\mu_i^{(2)}\) is \(\CC_i\) valued and hence commutes with \(e^{ik\th}\), and \(ij = -ji\) implies that \(e^{ik\th} j = j e^{-ik\th}\)).
	
	This is almost the inner product with respect to \(\mu_i\) that one might initially guess at, with the difference of a ``twist" in \(\p\) in the \(\mu_i^{(2)}\) term. To obtain the second form claimed, rewrite this as follows:
	\begin{align*}
	    \langle \p, \q \rangle_R & = \int_0^{2\pi} \overline{\q(e^{i\th})} \, \rmd\mu_i^{(1)}(\th) \p(e^{i\th}) + \int_0^{2\pi} \overline{\q(e^{i\th})} \, \rmd\mu_i^{(2)}(\th) j \p(e^{-i\th}) \\ 
	    & \quad \quad \quad + \left(\int_0^{2\pi} \overline{\q(e^{i\th})} \, \rmd\mu_i^{(2)}(\th) j \p(e^{i\th}) - \int_0^{2\pi} \overline{\q(e^{i\th})} \, \rmd\mu_i^{(2)}(\th) j \p(e^{i\th})\right) \\
	    & = \int_0^{2\pi} \overline{\q(e^{i\th})} \, \rmd\mu_i(\th) \p(e^{i\th}) + \int_0^{2\pi} \overline{\q(e^{i\th})} \, \rmd \mu_i^{(2)}(\th) j \left(\p(e^{-i\th}) - \p(e^{i\th}) \right).
	\end{align*}
\end{proof}

The discussion for right-slice hyperholomorphic polynomials is much the same.

\begin{definition}
	Let
	\[
	    \p(p) = \sum_{k=0}^n \p_k p^k, \quad \q(p) = \sum_{\ell = 0}^m \q_k p^k \in \HH[p]^R
	\]
	and define \(\hat{\p}, \hat{\q}\) as in \Cref{def:R-IP}. Then
	\[
	    \langle \p, \q \rangle_L := \hat{\p}^{\intercal} T_n(c) \overline{\hat{\q}}
	\]
	is a positive definite sesquilinear form.
\end{definition}

With a very similar calculation, one arrives at the following result.

\begin{theorem}
	Let \(\p, \q \in \HH[p]^R\). Then we may rewrite \( \langle \p, \q \rangle_L\) as
	\[
		\langle \p, \q \rangle_L = \int_0^{2\pi} \p(e^{i\th}) \, \rmd\mu_i^{(1)}(\th) \overline{\q(e^{i\th})} + \int_0^{2\pi} \p(e^{i\th}) \, \rmd\mu_i^{(2)}(\th) j \overline{\q(e^{-i\th})},
	\]
	and once more this can be rewritten as a perturbation of the ``expected" sesquilinear form as
	\[
		\langle \p, \q \rangle_L = \int_0^{2\pi} \p(e^{i\th}) \, \rmd\mu_i(\th) \overline{\q(e^{i\th})} + \int_0^{2\pi} \p(e^{i\th}) \, \rmd\mu_i^{(2)}(\th) j \left(\overline{\q(e^{-i\th}) - \q(e^{i\th})}\right).
	\]
\end{theorem}

The proof is almost identical to that of \Cref{thm:R-IP}.

In the complex-matricial setting, we shall work with the following sesquilinear forms, or ``matrix-valued inner products", which are linear in the first entry.

\begin{definition}
    Given a matrix-valued measure \(M_i\) on \(\TT_i\), we define the matrix-valued inner products \(\langle\cdot,\cdot\rangle_L\) and \(\langle\cdot,\cdot\rangle_R\) on \(f,g \in \CC_i^{2\times2}[z]\) by
    \[
        \langle f, g \rangle_L := \int_0^{2\pi} f(e^{i\th}) \, \rmd M_i(\th) g(e^{i\th})^*
    \]
    and
    \[
        \langle f, g \rangle_R := \int_0^{2\pi} g(e^{i\th})^* \, \rmd M_i(\th) f(e^{i\th}).
    \]
\end{definition}

\begin{remark}
    If we denote by \(\llangle \cdot, \cdot \rrangle_L\) and \(\llangle \cdot, \cdot \rrangle_R\) the matrix-valued inner products on \(\CC^{2\times2}[z]\) with respect to a positive \(2 \times 2\) matrix-valued measure \(M_i\) defined in \cite{DPS08}, which are linear in the second entry, then these are related to the forms of the previous definition by
    \[
        \langle f, g \rangle_L = \llangle g, f \rrangle_L
    \]
    and
    \[
        \langle f, g \rangle_R = \llangle g, f \rrangle_R.
    \]
\end{remark}

We are now in a position to define the orthonormal polynomials associated to a non-trivial positive definite sequence (or to a \(q\)-positive measure).

\begin{definition}
    Let \(i,j \in \SS\) be orthogonal and let \(c\) be a non-trivial positive definite sequence in \(\HH\) with associated \(q\)-positive measure \(\mu_i\) on \(\TT_i\). The \emph{orthonormal polynomials} associated to \(c\) (or to \(\mu_i\)) are the two sequences obtained via the Gram-Schmidt algorithm in with the coefficients on the right and left, using the inner products defined above. Gram-Schmidt with the coefficients on the right via the inner product \(\langle \cdot, \cdot \rangle_L\) produces the \emph{left-orthonormal polynomials}  \((\q_n^L)_{n=0}^{\infty}\), and Gram-Schmidt with the coefficients on the left and the inner product \(\langle \cdot, \cdot \rangle_R\) produces the \emph{right-orthonormal polynomials} \((\q_n^R)_{n=0}^{\infty}\). 
\end{definition}

We refer the reader to \cite[Theorem 4.3]{FP03} for a discussion of how the Gram-Schmidt algorithm extends to the quaternions.

\begin{remark}
	Note that the \emph{left} orthonormal polynomials are \emph{right}-slice hyperholomorphic, and vice versa.
\end{remark}

Under this regime, there exists a strong relationship between orthonormal polynomials in the quaternionic, \(q\)-positive measure setting and the matrix-valued measure setting.

\begin{theorem}
    Let \(i,j\in\SS\) be orthogonal and let \(M_i\) be a \(2\times2\) matrix-valued measure on \(\TT_i\), such that the moments \((C_n)_{n=-\infty}^{\infty}\) of \(M_i\) lie in the image of \(\x_i\). Let \((\p_n^L)_{n=0}^{\infty}\) and \((\p_n^R)_{n=0}^{\infty}\) be the matrix-valued orthonormal polynomials with respect to \(M_i\). If \(c = (\x_i^{-1}(C_n))_{n=-\infty}^{\infty}\) has quaternionic orthonormal polynomials \((\q_n^L)_{n=0}^{\infty}\) and \((\q_n^R)_{n=0}^{\infty}\), then for all \(n\),
    \[
        \p_n^L = \P_i^R(\q_n^L) \quad \text{ and } \quad \p_n^R = \P_i^L(\q_n^R).
    \]
\end{theorem}
\begin{proof}
    As discussed in the previous section, it follows from \Cref{thm:MatrixVerForm} that since \((C_n)_{n=-\infty}^{\infty}\) lies in the image of \(\x_i\), so too must the Verblunsky coefficients \((\a_n)_{n=0}^{\infty}\) of \(M_i\). From there, a simple induction and the matrix Szeg\H{o} recurrences \cite[Theorem 3.3]{DPS08} show that so too do all the coefficients of \(\p_n^L\) and \(\p_n^R\) for all \(n\). 
    
    Let \(\p_n^R\) have \(k^{\mathrm{th}}\) coefficient \(\p_{n,k}^R\), define quaternions \(\pi_{n,k}^R = \x_i^{-1}(\p_{n,k}^R)\) and define a sequence of quaternionic polynomials via
    \[
        \pi_n^R(p) = \sum_{k=0}^n p^k \pi_{n,k}^R \in \HH[p]^L
    \]
    so that \(\p_n^R = \P_i^L(\pi_n^R)\).
    
    Orthonormality of \((\p_n^R)_{n=0}^{\infty}\) implies that, for any \(n,m \in \NN\),
    \begin{align*}
        \d_{nm} I_2 & = \langle \p_n^R, \p_m^R \rangle_R \\
        & = \int_0^{2\pi} \left( \sum_{\ell = 0}^m \p_{m,\ell}^R e^{i\ell\th} \right)^* \, \rmd M_i(\th) \left(\sum_{k=0}^n \p_{n,k}^R e^{ik\th} \right) \\
        & = \sum_{\ell = 0}^m \sum_{k=0}^n (\p_{m,\ell}^R)^* \left( \int_0^{2\pi} e^{i(k-\ell)\th} \, \rmd M_i(\th) \right) \p_{n,k}^R \\
        & = \sum_{\ell = 0}^m \sum_{k=0}^n \x_i(\overline{\pi_{m,\ell}^R}) \left( \int_0^{2\pi} e^{i(k-\ell)\th} \, \rmd M_i(\th) \right) \x_i(\pi_{n,k}^R) \\
        & = \x_i\left( \sum_{\ell=0}^{m} \sum_{k=0}^n \overline{\pi_{m,\ell}^R} \cdot c_{k-\ell} \cdot \pi_{n,k}^R\right)
    \end{align*}
    and hence
    \[
        \d_{nm} = \sum_{\ell=0}^{m} \sum_{k=0}^n \overline{\q_{m,\ell}^R} \cdot c_{k-\ell} \cdot \q_{n,k}^R.
    \]
    Observe that the right hand side here is precisely \(\langle \q_n^R, \q_m^R \rangle_R\), so that the \(\HH[p]^L\)-polynomials \((\pi_n^R)_{n=0}^{\infty}\) are orthonormal with respect to \(\langle \cdot, \cdot \rangle_R\). It remains to note that the output of the Gram-Schmidt algorithm is unique, i.e. there is only one such sequence --- namely, \((\q_n^R)_{n=0}^{\infty}\), and so
    \[
        \p_n^R = \P_i^L(\pi_n^R) = \P_i^L(\q_n^R).
    \]
    
    Very similarly, if we let \(\p_n^L\) have \(k^{\mathrm{th}}\) coefficient \(\p_{n,k}^L\) and write \(\p_{n,k}^L = \x_i(\pi_{n,k}^L)\) then orthonormality of \((\p_n^L)\) implies that
    \[
        \d_{nm} = \sum_{\ell=0}^m \sum_{k=0}^n \pi_{n,k}^L \cdot c_{k-\ell} \cdot \overline{\pi_{m,\ell}^L}
    \]
    and it follows that the \(\HH[p]^R\)-polynomials
    \[
        \pi_n^L(p) := \sum_{k=0}^n \pi_{n,k}^L p^k,
    \]
    i.e. such that \(\p_n^L = \P_i^R(\pi_n^L)\), are orthonormal with respect to \(\langle \cdot, \cdot \rangle_L\). Again, by uniqueness of the output of the Gram-Schmidt algorithm, there is only one such sequence: in this case, \((\q_n^L)_{n=0}^{\infty}\), and it follows that
    \[
        \p_n^L = \P_i^R(\pi_n^L) = \P_i^R(\q_n^L).
    \]
\end{proof}

This connection between orthonormal polynomials in the quaternionic and matricial settings will prove to be highly exploitable throughout the remainder of the paper.

\section{Quaternionic Szeg\H{o} Recurrences}

In this section we use the previous discussion to derive quaternionic analogues of the Szeg\H{o} recurrences for both the left and right orthonormal polynomials.

Recall that the reverse polynomial of a matrix-valued polynomial \(P(z) = \sum_k P_k z^k \in \CC_i^{2\times2}[z]\) is given by
\begin{align*}
    P^{\#}(z) & := z^n P(1/\overline{z})^* \\
    & = z^n (\sum_k P_k^* \cdot 1/z^k) \\
    & = \sum_k P_k^* z^{n-k}.
\end{align*}

Abstracting this to the quaternionic setting becomes delicate due to the noncommutativity between the coefficients and the variable. We shall make use of the following definition.

\begin{definition}
	Let \(\p(p) = \sum_k p^k \p_k \in \HH[p]^L\) and \(\q(p) = \sum_k \q_k p^k \in \HH[p]^R\) have degree \(n\) and \(m\) respectively; we define the reverse polynomial of \(\p\) to be
	\[
	    \p^{\#}(p) = \overline{\p(1/\overline{p})} \cdot p^n \in \HH[p]^R
	\]
	and the reverse of \(\q\) to be
	\[
	    \q^{\#}(p) = p^m \overline{\q(1/\overline{p})} \in \HH[p]^L.
	\]
\end{definition}

Intuitively, we are appealing to a careful extension of the usual formula to define a correspondence between \(\HH[p]^L\) and \(\HH[p]^R\). This definition allows us to obtain the following result.

\begin{theorem}
	Let \(c = (c_n)_{n=-\infty}^{\infty}\) be a non-trivial positive definite sequence with \(c_0 = 1\), let \(i \in \SS\) and let \(\mu_i\) be the \(q\)-positive measure on \(\TT_i\) with moments \(c\). Let \(M_i\) be the \(\CC_i^{2\times2}\)-valued measure with moments \((\x_i(c_n))_{n=-\infty}^{\infty}\). If the orthonormal polynomials of \(\mu_i\) are \((\q_n^L)_n\) and \((\q_n^R)_n\) and the matrix orthonormal polynomials of \(M_i\) are \((\p_n^L)_n\) and \((\p_n^R)_n\), then for all \(n \in \NN\),
	\[
		\p_n^{L,\#} = \P_i^L(\q_n^{R,\#}) \quad \text{ and } \quad \p_n^{R,\#} = \P_i^R(\q_n^{R,\#}).
	\]
\end{theorem}
\begin{proof}
	Both claims may be verified by direct computation. Firstly,

	\begin{align*}
	    \q_n^{L,\#}(p) & = p^n \overline{\q_n^L(1/\overline{p})} \\
	    & = p^n\left(\sum_{k=0}^n \frac{1}{p^k} \cdot \overline{\q_{n,k}^L}\right) \\
	    & = \sum_{k=0}^n p^{n-k} \overline{\q_{n,k}^L} \\
	    & = (\P_i^L)^{-1}(\p_n^{L,\#})(p),
	\end{align*}

	and secondly,
	
	\begin{align*}
	    \q_n^{R,\#}(p) & = \overline{\q_n^R(1/\overline{p})} p^n \\
	    & = \left( \sum_{k=0}^n \overline{\q_{n,k}^R} \cdot \frac{1}{p^k} \right) p^n \\
	    & = \sum_{k=0}^{n} \overline{\q_{n,k}^R} p^{n-k} \\
	    & = (\P_i^R)^{-1}(\p_n^{R,\#})(p).
	\end{align*}
\end{proof}

We are now in a position to obtain the quaternionic Szeg\H{o} recurrences.

\begin{theorem}
\label{thm:QuaternionicSzegoRecs}
	Let \(\mu_i\) be a non-trivial \(q\)-positive measure with Verblunsky coefficients \((\g_n)_{n=0}^{\infty}\) and orthonormal polynomials \((\q_n^L)_n\) and \((\q_n^R)_n\). Defining \(r_n := \sqrt{1 - \lvert \g_n \rvert^2}\), we have
	\begin{align*}
	    & \q_n^L \cdot p - r_n \q_{n+1}^L = \overline{\g}_n \q_n^{R,\#} \\
	    & p \cdot \q_n^R - \q_{n+1}^R r_n = \q_n^{L,\#} \overline{\g}_n.
	\end{align*}
\end{theorem}
\begin{proof}
	Let \(M_i\) be the matrix-valued measure whose moments are \(\x_i\) applied to those of \(\mu_i\). Let \((\a_n)_{n=0}^{\infty}\) be the Verblunsky coefficients of \(M_i\) with defects \(\r_n := (I - \a_n\a_n^*)^{\frac12} = (I - \a_n^* \a_n)^{\frac12}\), and let \((\p_n^L)_n\) and \((\p_n^R)_n\) be the matrix orthonormal polynomials of \(M_i\). Recall the matrix Szeg\H{o} recurrences of \cite{DPS08}:
	\begin{align*}
	    & z\p_n^L - \r_n \p_{n+1}^L = \a_n^* \q_n^{R,\#} \\
	    & z\p_n^R - \p_{n+1}^R \r_n = \p_n^{L,\#} \a_n^*.
	\end{align*}

	Observe firstly that in the first relation, all matricial constants are multiplied on the left, while in the second, all constants appear multiplied on the right, and recall from Section 3 that \(\x_i(\a_n) = \g_n\). We may thus apply \((\P_i^R)^{-1}\) to the first and \((\P_i^L)^{-1}\) to the second; doing so, we obtain the claimed relations.
\end{proof}

We may reverse both of the quaternionic Szeg\H{o} recurrences and rearrange to obtain a pair of recurrences for each of the left and right orthogonal polynomials, directly reminiscent of the classical Szeg\H{o} recurrences.

\begin{corollary}
\label{cor:PairedSzegoRecs}
	The quaternionic orthonormal polynomials obey the following paired recurrences: for the left polynomials,
	\[
	    \q_{n+1}^L(p) = r_n^{-1}(\q_n^L(p) \cdot p - \overline{\g_n} \cdot \q_n^{R,\#}(p)), \quad \q_{n+1}^{L,\#}(p) = r_n^{-1}(\q_n^{L,\#}(p) - p \cdot \q_n^R(p) \cdot \g_n),
	\]
	and for the right,
	\[
	    \q_{n+1}^R(p) = r_n^{-1}(p \cdot \q_n^R(p) - \q_n^{L,\#} \cdot \overline{\g_n}), \quad \q_{n+1}^{R,\#}(p) = r_n^{-1}(\q_n^{R,\#}(p) - \g_n \q_n^L(p)\cdot p).
	\]
\end{corollary}

\section{A Zeros Theorem for Quaternionic Orthonormal Polynomials}

This section contains a discussion of the Zeros Theorem in the quaternionic setting. The classical (that is, complex function-theoretic) result known as the Zeros Theorem appears in \cite{Sim05a} as Theorem 1.7.1, and subsequently Simon provides six proofs showcasing different ideas in the theory of orthogonal polynomials. A matrix-valued analogue appears in \cite{DPS08} as Theorem 3.7, and we shall utilise this result in this section.

\begin{definition}
    The \emph{zero set} of a (left- or right-) slice hyperholomorphic polynomial \(\p\) is
    \[
        \cZ(\p) := \{q \in \HH : \p(q) = 0\}.
    \]
\end{definition}

\begin{definition}
    The \emph{determinantal zero set} of a complex polynomial \(\p\) with \(d \times d\) matrix coefficients is 
    \[
        \cZ_{\det}(\p) := \{z \in \CC_i : \p(0) = 0_{d\times d}\}.
    \]
\end{definition}

\begin{theorem}
    \label{thm:leftZeroSet}
    Let \(\Q(p) = p^n + \sum_{\ell = 0}^{n-1} p^\ell \Q_\ell\) be a monic, left-slice hyperholomorphic polynomial. Then for \(i \in \SS\), the slice of \(\cZ(\Q)\) lying in \(\CC_i\) is the determinantal zero set of the matrix-valued polynomial \(\P^L_i(\Q)\).
\end{theorem}
\begin{proof}
    The companion matrix of \(\Q\) is, as discussed in \cite{SPV01},
    \[
        \cC_\Q := \begin{bmatrix} 0 & \cdots & 0 & -\Q_0 \\ 1 & \ddots & \vdots & -\Q_1 \\ & \ddots & 0 & \vdots \\ 0 & & 1 & -\Q_{n-1} \end{bmatrix}
    \]
    and (see Proposition 1 as well as the remark at the end of Section 5 of the same source) this matrix has the property that \(\cZ(\Q) = \s_L(\cC_\Q)\). Moreover, the eigenvectors of \(\cC_\Q\) take the form
    \[
        \begin{bmatrix} 1 \\ \l \\ \vdots \\ \l^n \end{bmatrix}, \quad \l \in \s_L(\cC_\Q);
    \]
    notice that any such vector commutes with its corresponding eigenvalue, so that \(\s_L(\cC_\Q) = \s_R(\cC_\Q)\).
    
    Now, by \Cref{lem:UnitaryEquiv}, there exists some unitary \(U\) so that 
    \[
        U^* \x_i(\cC_\Q) U = \begin{bmatrix} 0 & \cdots & 0 & -\x_i(\Q_0) \\ I_2 & \ddots & \vdots & -\x_i(\Q_1) \\ & \ddots & 0 & \vdots \\ 0 & & I_2 & -\x_i(\Q_{n-1}) \end{bmatrix},
    \]
    and the spectrum of this matrix is the determinantal zero set of the matrix-valued polynomial
    \[
        \P(z) = I_2 \cdot z^n + \sum_{\ell = 0}^{n-1} \x_i(\Q_\ell) z^\ell
    \]
    (see \cite{GLR09}). Using that for any \(i \in \SS\), \(\x_i(1) = I_2\), we have that
    \[
        \P^L_i(\Q)(z) = \P(z).
    \]
    On the other hand, the spectrum of \(\x_i(\cC_{\Q})\) is the slice of the right-spectrum of \(\cC_\Q\) lying in \(\CC_i\): this follows from the remark of Zhang \cite{Zha97} that if \(\l \in \CC_i\), the quaternionic right-eigenvalue equation
    \[
        A\xi = \xi \l
    \]
    is equivalent to the complex eigenvalue equation
    \[
        \x_i(A) \begin{bmatrix} \xi^{(1)} \\[6pt] - \overline{\xi^{(2)}}\end{bmatrix} = \l \begin{bmatrix} \xi^{(1)} \\[6pt] - \overline{\xi^{(2)}}\end{bmatrix},
    \]
    which has the consequence that for a quaternionic matrix \(A\),
    \[
        \l \in \CC_i \cap \s_R(A) \quad \text{ if and only if } \quad \l \in \s(\x_i(A)).
    \]
    Of course, we also have by multiplicativity of the determinant that \(\s(\x_i(\cC_\Q)) = \s(U \x_i(\cC_\Q) U^*)\).
    
    To complete the proof, simply string together our established equalities:
    \[
        \cZ(\Q) \cap \CC_i = \s_L(\cC_\Q) \cap \CC_i = \s_R(\cC_\Q) \cap \CC_i = \s(\x_i(\cC_\Q)) = \s(U^* \x_i(\cC_\Q) U) = \cZ_{\det}(\P) = \cZ_{\det}(\P^L_i(\Q)).
    \]
\end{proof}

\begin{corollary}
    \label{cor:rightPolyZeros}
    Let \(c = (c_n)_{n=-\infty}^{\infty} \subseteq \HH\) be a non-trivial positive definite sequence with \(c_0 = 1\) and let the right-orthonormal polynomials associated to \(c\) be \((\q_n^R)_n\). Let \(M_i\) be the positive matrix-valued measure with moments \((\x_i(c_n))_{n=-\infty}^{\infty}\) and right-orthonormal polynomials \((\p_n^R)_n\). Then for any \(i \in \SS\),
    \[
        \cZ(\q_n^R) \cap \CC_i = \cZ_{\det}(\p_n^R),
    \]
    and hence
    \[
        \cZ(\q_n^R) = \bigcup_{i \in \SS} \cZ_{\det}(\p_n^R).
    \]
\end{corollary}
\begin{proof}
    Given \(i \in \SS\), let the leading coefficients of \(\q_n^R\) and \(\p_n^R\) be \(\q_{n,n}^R \in \HH\) and \(\p_{n,n}^R \in \CC_i^{2\times 2}\), respectively. Then \(\Q_n^R := (\q_{n,n}^R)^{-1}\q_n^R\) is a monic, left-slice hyperholomorphic polynomial, and similarly \(\P_n^R = (\p_{n,n}^R)^{-1}\p_n^R\) is a monic matrix-valued polynomial. Moreover, we have that
    \[
        \cZ(\q_n^R) = \cZ(\Q_n^R) \quad \text{ and } \quad \cZ_{\det}(\p_n^R) = \cZ_{\det}(\P_n^R).
    \]
    Now, the sequence of matrix-valued polynomials
    \[
        \P_i^L(\Q_n^R), \quad n \in \NN
    \]
    is by construction right-orthogonal with respect to \(M_i\). As is remarked in \cite{DPS08}, such sequences are unique, and hence \(\P_n^R = \P_i^L(\Q_n^R\). The first claim now follows from \Cref{thm:leftZeroSet}.

    To see the second claim, recall that every quaternion lies in \(\CC_i\) for some \(i \in \SS\) and take a union over \(i\) of the first claim.
\end{proof}

The corresponding result for monic, right-slice hyperholomorphic polynomials also holds, and accordingly we obtain a similar corollary for the left-orthonormal polynomials of a non-trivial positive definite sequence.

\begin{theorem}
    \label{thm:rightZeroSet}
    Let \(\Q(p) = p^n + \sum_{\ell = 0}^{n-1} \Q_\ell p^\ell\) be a monic, right-slice hyperholomorphic polynomial. Then for \(i \in \SS\), the slice of \(\cZ(\Q)\) lying in \(\CC_i\) is the determinantal zero set of the matrix-valued polynomial \(\P^R_i(\Q)\).
\end{theorem}
\begin{proof}
    The proof is essentially identical to that of \Cref{thm:leftZeroSet}. Referring again to \cite[Proposition 1]{SPV01}, the only modification that must be made is that the appropriate companion matrix is now
    \[
        \cC_\Q := \begin{bmatrix}
            0 & 1 & & 0 \\ \vdots & \ddots & \ddots & & \\ 0 & \cdots & 0 & 1 \\ -\Q_0 & -\Q_1 & \cdots & -\Q_{n-1}
        \end{bmatrix}.
    \]
\end{proof}

\begin{corollary}
    \label{cor:leftPolyZeros}
    Let \(c = (c_n)_{n=-\infty}^{\infty} \subseteq \HH\) be a non-trivial positive definite sequence with \(c_0 = 1\) and let the left-orthonormal polynomials associated to \(c\) be \((\q_n^L)_n\). Let \(M_i\) be the positive matrix-valued measure with moments \((\x_i(c_n))_{n=-\infty}^{\infty}\) and left-orthonormal polynomials \((\p_n^L)_n\). Then for any \(i \in \SS\),
    \[
        \cZ(\q_n^L) \cap \CC_i = \cZ_{\det}(\p_n^L),
    \]
    and hence
    \[
        \cZ(\q_n^L) = \bigcup_{i \in \SS} \cZ_{\det}(\p_n^L).
    \]
\end{corollary}
\begin{proof}
    We proceed as in the proof of \Cref{cor:rightPolyZeros}, with \Cref{thm:rightZeroSet} in place of \Cref{thm:leftZeroSet}.
\end{proof}

\begin{remark}
    \label{rem:zeroSetsEqual}
    It now follows from \cite[Theorem 3.8]{DPS08} that
    \[
        \cZ(\q_n^R) = \bigcup_{i \in \SS} \cZ_{\det}(\p_n^R) = \bigcup_{i\in\SS} \cZ_{\det}(\p_n^L) = \cZ(\q_n^L),
    \]
    i.e. while the left and right orthonormal polynomials are different objects, their zero sets coincide.
\end{remark}    

We finally obtain from the above and \cite[Theorem 3.7]{DPS08} the quaternionic analogue of the Zeros Theorem: the zero sets of both the left- and right-orthonormal polynomials lie strictly inside the quaternionic unit ball, while the zero sets of their reverse polynomials lie strictly outside the closed unit ball.

\begin{theorem}
\label{thm:QuaternionicZerosThm}
    Let \(c \subseteq \HH\) be a non-trivial positive definite sequence with associated left- and right-orthonormal polynomials \((\p_n^L)_n\) and \((\p_n^R)_n\), and denote their respective reverse polynomials by \((\p_n^{L,\#})_n\) and \((\p_n^{R,\#})_n\). We have for \(n \in \NN\) that firstly,
    \[
        \cZ(\q_n^L) \subseteq \BB \quad \text{ and } \quad \cZ(\q_n^R) \subseteq \BB,
    \]
    and secondly,
    \[
        \cZ(\q_n^{L,\#}) \subseteq \HH \setminus \overline{\BB} \quad \text{ and } \quad \cZ(\q_n^{R,\#}) \subseteq \HH \setminus \overline{\BB}.
    \]
\end{theorem}
\begin{proof}
    By \Cref{cor:rightPolyZeros} and \Cref{cor:leftPolyZeros}, \(\cZ(\q_n^L) \subseteq \bigcup_{i \in \SS} \cZ_{\det}(\p_n^L)\) and \(\cZ(\q_n^R) \subseteq \bigcup_{i \in \SS} \cZ_{\det}(\p_n^R)\). We may combine these facts with \cite[Theorem 3.7]{DPS08} and \Cref{rem:zeroSetsEqual} to see that
    \[
        \cZ(\q_n^R) = \cZ(\q_n^L) \subseteq \bigcup_{i\in\SS} \DD_i = \BB.
    \]

    The reverse formulae for left- and right-slice hyperholomorphic polynomials imply that \(\q_n^L(p) = 0\) if and only if \(\q_n^{L,\#}(1/\overline{p}) = 0\), and similarly for \(\q_n^R\); applying this to the first claim proves the second.
\end{proof}

\section{A Diagonal Christoffel--Darboux Formula}

This section applies the quaternionic Szeg\H{o} recurrences to obtain an analogue on the diagonal of the well-known Christoffel--Darboux formula.

\begin{theorem}
\label{thm:CDFDiagonal}
	Let \(c \subseteq \HH\) be a non-trivial positive definite sequence with orthonormal polynomials  \((\q_n^L)_n\) and \((\q_n^R)_n\). Then for \(p \in \HH\setminus\partial\BB\),
	\begin{align*}
		\sum_{\ell = 0}^n \lvert \q_\ell^L(p) \rvert^2 + \lvert \q_\ell^R(p) \rvert^2 & = \frac{\lvert \q_{n+1}^{L.\#}(p) \rvert^2 + \lvert \q_{n+1}^{R,\#}(p)\rvert^2 - \Big(\lvert \q_{n+1}^L(p) \rvert^2 + \lvert \q_{n+1}^R(p) \rvert^2\Big)}{1 - \lvert p \rvert^2} \\ 
		& = \frac{\lvert \q_n^{L,\#}(p) \rvert^2 + \lvert \q_n^{R,\#}(p) \rvert^2 - \lvert p \rvert^2 \Big(\lvert \q_n^L(p)\rvert^2 + \lvert \q_n^R(p) \rvert^2\Big)}{1 - \lvert p \rvert^2}.
	\end{align*}
\end{theorem}
\begin{proof}
Using the rearranged Szeg\H{o} recurrences, we may compute directly that for \(p,q \in \HH\),
\begin{multline*}
    \overline{\q_{n+1}^{L,\#}(p)} \cdot \q_{n+1}^{L,\#}(q) - \overline{\q_{n+1}^L(p)} \cdot \q_{n+1}^L(q) = \\ r_n^{-2}\Bigg[\overline{\q_n^{L,\#}(p)} \q_n^{L,\#}(q) - \overline{\q_n^{L,\#}(p)} q \q_n^R(q) \g_n - \overline{\g_n} \overline{\q_n^R(p)} \overline{p} \q_n^{L,\#}(q) + \overline{\g_n} \overline{\q_n^R(p)} \overline{p}q \q_n^R(q) \g_n \\ - \overline{p} \overline{\q_n^L(p)} \q_n^L(q) q + \overline{p} \overline{\q_n^L(p)} \overline{\g_n} \q_n^{R,\#}(q) + \overline{\q_n^{R,\#}(p)} \g_n \q_n^L(q) q - \overline{\q_n^{R,\#}(p)} \lvert \g_n \rvert^2 \q_n^{R,\#}(q) \Bigg].
\end{multline*}
Setting \(p = q\), this simplifies on the diagonal to
\begin{multline*}
    \lvert \q_{n+1}^{L.\#}(p) \rvert^2 - \lvert \q_{n+1}^L(p) \rvert^2 = r_n^{-2} \Bigg[ \lvert \q_n^{L,\#}(p) \rvert^2 -  \lvert p \rvert^2 \lvert \q_n^L(p)\rvert^2 - 2\Re\Big(\overline{\q_n^{L,\#}(p)} \, p \, \q_n^R(p) \, \g_n\Big) \\ + 2\Re\Big( \overline{\q_n^{R,\#}(p)} \, \g_n \, \q_n^L(p) \, p \Big) - \lvert \g_n \rvert^2 \Big( \lvert \q_n^{R,\#}(p) \rvert^2 - \lvert \q_n^R(p) \rvert^2 \lvert p \rvert^2  \Big) \Bigg].
\end{multline*}

Repeating this calculation for the right polynomials, we obtain
\begin{multline*}
    \overline{\q_{n+1}^{R,\#}(p)} \q_{n+1}^{R,\#}(q) - \overline{\q_{n+1}^R(p)} \q_{n+1}^R(q) = \\ r_n^{-2}\Bigg[\overline{\q_n^{R,\#}(p)} \q_n^{R,\#}(q) - \overline{\q_n^{R,\#}(p)} \g_n \q_n^L(q)q - \overline{p} \overline{\q_n^L(p)} \overline{\g_n} \q_n^{R,\#}(q) + \overline{p} \overline{\q_n^L(p)} \lvert \g_n \rvert^2 \q_n^L(q) q \\ - \overline{\q_n^R(p)} \overline{p} q \q_n^R(q) + \overline{\q_n^R(p)} \overline{p} \q_n^{L,\#}(q) \overline{\g_n} + \g_n \overline{\q_n^{L,\#}(p)} q \q_n^R(q) - \g_n \overline{\q_n^{L,\#}(p)} \q_n^{L,\#}(q) \overline{\g_n} \Bigg]
\end{multline*}
and hence
\begin{multline*}
    \lvert \q_{n+1}^{R,\#}(p)\rvert^2 - \lvert \q_{n+1}^R(p) \rvert^2 = r_n^{-2}\Bigg[ \lvert \q_n^{R,\#}(p) \rvert^2 - \lvert p \rvert^2 \lvert \q_n^R(p) \rvert^2 - 2\Re\Big(\overline{\q_n^{R,\#}(p)} \, \g_n \, \q_n^L(p) \, p \Big) \\ + 2\Re\Big( \g_n \, \overline{\q_n^{L,\#}(p)} \, p \, \q_n^R(p)\Big) - \lvert \g_n \rvert^2 \Big(\lvert \q_n^{L,\#}(p) \rvert^2 - \lvert p \rvert^2 \lvert \q_n^L(p) \rvert^2 \Big) \Bigg].
\end{multline*}

Summing these equalities on the diagonal, we find some pleasing cancellation. Further using the fact that for quaternions \(r,s\in\HH\) one has \(\Re(sr) = \Re(rs)\), we arrive at
\begin{align*}
    & \lvert \q_{n+1}^{L.\#}(p) \rvert^2 - \lvert \q_{n+1}^L(p) \rvert^2 + \lvert \q_{n+1}^{R,\#}(p)\rvert^2 - \lvert \q_{n+1}^R(p) \rvert^2 = \\ & \quad r_n^{-2} \Bigg[ \lvert \q_n^{L,\#}(p) \rvert^2 -  \lvert p \rvert^2 \lvert \q_n^L(p)\rvert^2 + \lvert \q_n^{R,\#}(p) \rvert^2 - \lvert p \rvert^2 \lvert \q_n^R(p) \rvert^2 \\ & \quad - \lvert \g_n \rvert^2 \Big( \lvert \q_n^{R,\#}(p) \rvert^2 - \lvert \q_n^R(p) \rvert^2 \lvert p \rvert^2 + \lvert \q_n^{L,\#}(p) \rvert^2 - \lvert p \rvert^2 \lvert \q_n^L(p) \rvert^2 \Big) \Bigg] \\
    & = (1 - \lvert \g_n \rvert^2)^{-1} (1 - \lvert \g_n \rvert^2) \Bigg[ \lvert \q_n^{L,\#}(p) \rvert^2 + \lvert \q_n^{R,\#}(p) \rvert^2 - \lvert p \rvert^2 \lvert \q_n^L(p)\rvert^2 - \lvert p \rvert^2 \lvert \q_n^R(p) \rvert^2 \Bigg] \\
    & = \lvert \q_n^{L,\#}(p) \rvert^2 + \lvert \q_n^{R,\#}(p) \rvert^2 - \lvert p \rvert^2 \lvert \q_n^L(p)\rvert^2 - \lvert p \rvert^2 \lvert \q_n^R(p) \rvert^2,
\end{align*}
i.e.
\[
    \lvert \q_{n+1}^{L.\#}(p) \rvert^2 + \lvert \q_{n+1}^{R,\#}(p)\rvert^2 - \Big(\lvert \q_{n+1}^L(p) \rvert^2 + \lvert \q_{n+1}^R(p) \rvert^2\Big) = \lvert \q_n^{L,\#}(p) \rvert^2 + \lvert \q_n^{R,\#}(p) \rvert^2 - \lvert p \rvert^2 \Big(\lvert \q_n^L(p)\rvert^2 + \lvert \q_n^R(p) \rvert^2\Big).
\]

This gives a quaternionic version of one half of the Christoffel--Darboux formula, restricted to the diagonal \(p = q\) (specifically, this is a direct quaternionic analogue of the equality (2.2.41) = (2.2.42) of \cite{Sim05a} with \(z = \z\)):
\begin{multline*}
    \frac{\lvert \q_{n+1}^{L.\#}(p) \rvert^2 + \lvert \q_{n+1}^{R,\#}(p)\rvert^2 - \Big(\lvert \q_{n+1}^L(p) \rvert^2 + \lvert \q_{n+1}^R(p) \rvert^2\Big)}{1 - \lvert p \rvert^2} \\ = \frac{\lvert \q_n^{L,\#}(p) \rvert^2 + \lvert \q_n^{R,\#}(p) \rvert^2 - \lvert p \rvert^2 \Big(\lvert \q_n^L(p)\rvert^2 + \lvert \q_n^R(p) \rvert^2\Big)}{1 - \lvert p \rvert^2},
\end{multline*}
for \(n \in \NN\) and \(p \in \HH\setminus\partial\BB\).

Now, set \(Q_n(p) := \frac{\lvert \q_n^{L,\#}(p) \rvert^2 + \lvert \q_n^{R,\#}(p) \rvert^2 - \lvert p \rvert^2 \Big(\lvert \q_n^L(p)\rvert^2 + \lvert \q_n^R(p) \rvert^2\Big)}{1 - \lvert p \rvert^2}\). Using the second form for \(Q_n\) and the first form for \(Q_{n-1}\), we have
\begin{align*}
    Q_n(p) - Q_{n-1}(p) & = \frac{1}{1 - \lvert p \rvert^2} (1 - \lvert p \rvert^2)\Big(\lvert \q_n^L(p) \rvert^2 + \lvert \q_n^R(p) \rvert^2\Big) \\
    & = \lvert \q_n^L(p) \rvert^2 + \lvert \q_n^R(p) \rvert^2 \\
    & = K_n(p) - K_{n-1}(p),
\end{align*}
where
\[
    K_n(p) := \sum_{\ell = 0}^n \lvert \q_\ell^L(p) \rvert^2 + \lvert \q_\ell^R(p) \rvert^2
\]
is an object playing the role in the quaternionic setting of the Christoffel--Darboux kernel, restricted to the diagonal.

Moreover, notice that \(Q_0(p) = \frac{1 + 1 - \lvert p \rvert^2 (1 + 1)}{1 - \lvert p \rvert^2} = 2 = K_0(p)\). An immediate induction on \(n\) demonstrates that \(Q_n(p) = K_n(p)\) for \(n \in \NN\) and \(p \in \HH\setminus\partial\BB\), and so we have a full analogue on the diagonal of the Christoffel--Darboux formula:
\begin{align*}
    \sum_{\ell = 0}^n \lvert \q_\ell^L(p) \rvert^2 + \lvert \q_\ell^R(p) \rvert^2 & = \frac{\lvert \q_{n+1}^{L.\#}(p) \rvert^2 + \lvert \q_{n+1}^{R,\#}(p)\rvert^2 - \Big(\lvert \q_{n+1}^L(p) \rvert^2 + \lvert \q_{n+1}^R(p) \rvert^2\Big)}{1 - \lvert p \rvert^2} \\ 
    & = \frac{\lvert \q_n^{L,\#}(p) \rvert^2 + \lvert \q_n^{R,\#}(p) \rvert^2 - \lvert p \rvert^2 \Big(\lvert \q_n^L(p)\rvert^2 + \lvert \q_n^R(p) \rvert^2\Big)}{1 - \lvert p \rvert^2}.
\end{align*}
\end{proof}

\begin{remark}
    Rearrange the first formula of the previous result to see for \(n \in \NN\) and \(p \in \BB\) that
    \begin{equation*}
    \label{eq:CDDiag}
        \lvert \q_{n+1}^{L.\#}(p) \rvert^2 + \lvert \q_{n+1}^{R,\#}(p)\rvert^2 = \lvert \q_{n+1}^L(p) \rvert^2 + \lvert \q_{n+1}^R(p) \rvert^2 + (1 - \lvert p \rvert^2) \sum_{\ell = 0}^n (\lvert \q_\ell^L(p) \rvert^2 + \lvert \q_\ell^R(p) \rvert^2).
    \end{equation*}

    Once one proves that \(\cZ(\q_n^L) = \cZ(\q_n^R)\) as in the previous section, this equality may be used to provide an alternate, direct proof of the quaternionic Zeros Theorem, analogous to the fifth proof of Theorem 1.7.1 in \cite{Sim05a} but grounded fully in the quaternionic theory.
\end{remark}

\section{The Szeg\H{o}--Verblunsky Theorem}

In this section we obtain an analogue of the Szeg\H{o}--Verblunsky theorem for the quaternions.

\begin{theorem}
\label{thm:QuaternionicSV}
    Let \(c = (c_n)_{n=-\infty}^{\infty} \subset \HH\) be a non-trivial positive definite sequence with quaternionic Verblunsky coefficients \((\g_n)_{n=0}^{\infty}\), choose orthogonal \(i,j\in\SS\) and let \(\mu_i\) be the unique \(q\)-positive measure on \(\TT_i\) with moments \(c\); let the corresponding matrix-valued measure be \(M_i\) with Radon--Nikodym derivative \(W_i \in L^1(\TT_i; \CC_i^{2\times2})\). Then
    \begin{equation}
    \label{eq:quaternionicSV}
        \prod_{n=0}^{\infty} (1 - \lvert \g_n \rvert^2)^2 = \exp\left(\int_0^{2\pi} \log \det W_i(e^{i\th}) \, \frac{\rmd\th}{2\pi}\right).
    \end{equation}

    Moreover, for any other choice of orthogonal \(i',j' \in \SS\) with resulting matrix-valued measure \(M_{i'}\) with Radon--Nikodym derivative \(W_{i'}\), we have
    \[
        \int_0^{2\pi} \log \det W_i(e^{i\th}) \, \frac{\rmd\th}{2\pi} = \int_0^{2\pi} \log \det W_{i'}(e^{i' \th}) \, \frac{\rmd\th}{2\pi},
    \]
    i.e. the matricial Szeg\H{o} entropy of the matrix-valued measure corresponding to \(c\) is invariant under the choice of basis \(i,j\in\SS\).
\end{theorem}

\begin{proof}
    The matricial Szeg\H{o}--Verblunsky theorem due to Delsarte, Genin and Kamp \cite{DGK78} applied to \(M_i\) states that, if \(M_i\) has Verblunsky coefficients \(\a_n)_{n=0}^{\infty}\), then
    \[
        \prod_{n=0}^{\infty} \det (I - \a_n\a_n^*) = \exp\left(\int_0^{2\pi} \log \det W_i(e^{i\th}) \, \frac{\rmd\th}{2\pi}\right);
    \]
    recall once more that the Verblunsky coefficients of \(M_i\) are related to those of \(\mu_i\) by \(\a_n = \x_i(\g_n)\).
    
    Since \(1 - \lvert\g_n\rvert^2\) is a real number and \(\x_i\) is a \(*\)-homomorphism, we compute directly that
    \[
        \det( I - \a_n\a_n^* ) = \det \x_i(1 - \lvert \g_n \rvert^2) = \det \begin{bmatrix} 1 - \lvert \g_n \rvert^2 & 0 \\ 0 & 1 - \lvert \g_n \rvert^2 \end{bmatrix} = (1 - \lvert \g_n \rvert^2)^2,
    \]
    so that
    \[
        \prod_{n=0}^{\infty} (1 - \lvert \g_n \rvert^2)^2 = \exp\left(\int_0^{2\pi} \log \det W_i(e^{i\th}) \, \frac{\rmd\th}{2\pi}\right).
    \]

    The final claim follows by noting that, since \((\g_n)_{n=0}^{\infty}\) is a global object, that is, invariant under choice of \(i,j\in\SS\), we have
    \[
        \exp\left(\int_0^{2\pi} \log \det W_i(e^{i\th}) \, \frac{\rmd\th}{2\pi}\right) = \prod_{n=0}^{\infty} (1 - \lvert \g_n \rvert^2)^2 = \exp\left(\int_0^{2\pi} \log \det W_{i'}(e^{i' \th}) \, \frac{\rmd\th}{2\pi}\right).
    \]
\end{proof}

\begin{remark}
    The additional power of 2 on the left-hand side of \eqref{eq:quaternionicSV}, as compared to the classical theory, is an artefact of the particular embedding of the quaternions as a subalgebra of \(2 \times 2\) matrices over \(\CC\).
\end{remark}

As we discussed in the introduction, an important consequence of the Szeg\H{o}--Verblunsky theorem is that square-summability of the Verblunsky coefficients is equivalent to log-integrability of the Radon--Nikodym derivative. We obtain a similar condition from this result.

\begin{corollary}
\label{cor:SVConvergence}
    Let \(c = (c_n)_{n=-\infty}^{\infty}\) be a non-trivial positive definite sequence with quaternionic Verblunsky coefficients \((\g_n)_{n=0}^{\infty}\), choose orthogonal \(i,j\in\SS\) and let \(M_i\) be the matrix-valued measure with moments \((\x_i(c_n))_{n=-{\infty}}^{\infty}\) with Radon--Nikodym derivative \(W_i \in L^1(\TT_i; \CC_i^{2\times2})\). Then
    \[
        \sum_{n=0}^{\infty} \lvert \g_n \rvert^2 < \infty \quad \text{ if and only if } \quad \int_0^{2\pi} \log \det W_i(e^{i\th}) \, \frac{\rmd\th}{2\pi} > -\infty.
    \]
\end{corollary}
\begin{proof}
    First,
    \[
        \int_0^{2\pi} \log \det W_i(e^{i\th}) \, \frac{\rmd\th}{2\pi} > -\infty \quad \text{ if and only if } \quad \exp\left(\int_0^{2\pi} \log \det W_i(e^{i\th}) \, \frac{\rmd\th}{2\pi}\right) > 0,
    \]
    and via \Cref{thm:QuaternionicSV} the latter condition is equivalent
    \[
        \prod_{n=0}^{\infty} (1 - \lvert \g_n \rvert^2)^2 > 0.
    \]
    Rewriting \((1 - \lvert \g_n \rvert^2)^{2}\) as \(1 - 2\lvert \g_n \rvert^2 + \lvert \g_n \rvert^4 = 1 - \lvert \g_n \rvert^2 (2 - \lvert \g_n \rvert^2)\), a classical lemma of real analysis (as in e.g. \cite{Kat04} following Lemma 3.3.5) shows that this product converges and is nonzero if and only if
    \[
        \sum_{n=0}^{\infty} \lvert \g_n \rvert^2 (2 - \lvert \g_n \rvert^2) < \infty,
    \]
    so it remains to show that
    \[
        \sum_{n=0}^{\infty} \lvert \g_n \rvert^2 (2 - \lvert \g_n \rvert^2) < \infty \quad \text{ if and only if } \sum_{n=0}^{\infty} \lvert \g_n \rvert^2.
    \]

    Supposing that \(\sum_{n=0}^{\infty} \lvert \g_n \rvert^2 (2 - \lvert \g_n \rvert^2) < \infty\), notice that since \(\g_n \in \BB\) for all \(n\), we have that \(2 - \lvert \g_n \rvert^2 > 1\) and consequently \(\lvert \g_n \rvert^2 (2 - \lvert \g_n \rvert^2) > \lvert \g_n \rvert^2\) for all \(n\). It follows by comparison that
    \[
        \sum_{n=0}^{\infty} \lvert \g_n \rvert^2 < \sum_{n=0}^{\infty} \lvert \g_n \rvert^2 (2 - \lvert \g_n \rvert^2) < \infty.
    \]

    On the other hand, suppose that \(\sum_{n=0}^{\infty} \lvert \g_n \rvert^2 < \infty\). Again since \(\g_n \in \BB\) for all \(n\), we have that \(\lvert \g_n \rvert^4 < \lvert \g_n \rvert^2\), so once more by comparison we have that
    \[
        \sum_{n=0}^{\infty} \lvert \g_n \rvert^4 < \sum_{n=0}^{\infty} \lvert \g_n \rvert^2 < \infty.
    \]
    It follows that
    \[
        \sum_{n=0}^{\infty} \lvert \g_n \rvert^2 (2 - \lvert \g_n \rvert^2) = 2\sum_{n=0}^{\infty} \lvert \g_n \rvert^2 + \sum_{n=0}^{\infty} \lvert \g_n \rvert^4 < \infty,
    \]
    and the claim is proved.
    
\end{proof}

We complete the section with a discussion of how \Cref{thm:QuaternionicSV} recovers the classical Szeg\H{o}--Verblunsky theorem when all of the quaternionic Verblunsky coefficients lie in a particular complex plane.

\begin{remark}
    Notice that when all the Verblunsky coefficients of a positive definite sequence \(c = (c_n)_{n=-\infty}^{\infty}\) lie in the same complex plane, say \((\g_n)_{n=0}^{\infty} \subseteq \CC_i\), then by \Cref{thm:MatrixVerForm}, the sequence \(c\) itself must as well. If \(C_n := \x_i(c_n)\) and \(M_i\) is the matrix-valued measure with moments \((C_n)_{n=-\infty}^{\infty}\), it follows that the \(C_n\) are in fact diagonal matrices given by \(C_n = \begin{bmatrix} c_n & 0 \\ 0 & c_{-n} \end{bmatrix}\), and therefore the Radon--Nikodym derivative \(W_i\) of \(M_i\) takes values in diagonal matrices. Moreover, notice that under these conditions we similarly have \(\det \a_n = \det \x_i(\g_n) = \lvert \g_n \rvert^2\), where \((\a_n)_{n=0}^{\infty}\) are the Verblunsky coefficients of \(M_i\).

	If \(\mu\) is the (now complex) measure with moments \((c_n)_{n=-\infty}^{\infty}\) and \(\tilde{\mu}\) is the measure such that
	\[
		\int_0^{2\pi} e^{in\th} \, \rmd \tilde{\mu}(\th) = \int_0^{2\pi} e^{-in\th} \, \rmd\mu(\th)
	\]
	then the matrix-valued measure \(\begin{bmatrix} \mu & 0 \\ 0 & \tilde{\mu} \end{bmatrix}\) is a representing measure for the moment sequence \((C_n)_{n=-\infty}^{\infty}\); by the uniqueness of the solution to the moment problem (see e.g. \cite[Theorem 1]{DGK78}), this measure is therefore the measure \(M_i\).

	Writing the Lebesgue decomposition of \(\mu\) as \(\rmd \mu = w \frac{\rmd\th}{2\pi} + \rmd \mu_\rms\), it is clear that \(\tilde{\mu}\) has Lebesgue decomposition \(\rmd\tilde{\mu} = \tilde{w} \frac{\rmd\th}{2\pi} + \rmd\tilde{\mu}_\rms\) where \(\tilde{w}(e^{i\th}) = w(e^{i(2\pi - \th)})\), and so \(M_i\) decomposes as
	\[
		\rmd M_i = \begin{bmatrix} w & 0 \\ 0 & \tilde{w} \end{bmatrix} \frac{\rmd\th}{2\pi} + \begin{bmatrix} \rmd \mu_\rms & 0 \\ 0 & \rmd \tilde{\mu}_\rms \end{bmatrix}.
	\]
	Thus the diagonal matrix-valued function \(W_i\) is given by
	\[
		W_i = \begin{bmatrix}w & 0 \\ 0 & \tilde{w}\end{bmatrix}
	\]
    and the singular part of \(M_i\) is
    \[
        \begin{bmatrix} \rmd \mu_\rms & 0 \\ 0 & \rmd \tilde{\mu}_\rms \end{bmatrix},
    \]
    see \cite[Section 2]{DGK78} and \cite[pg. 4]{Hof62}.
    
	With this in hand, the quaternionic Szeg\H{o}--Verblunsky theorem becomes
	\begin{align*}
		\prod_{n=0}^{\infty} (1 - \lvert \g_n \rvert^2)^2 & = \exp\left(\int_0^{2\pi} \log \det W_i(e^{i\th}) \, \frac{\rmd\th}{2\pi}\right) \\
		& = \exp\left(\int_0^{2\pi} \log w(e^{i\th}) \, \frac{\rmd\th}{2\pi} + \int_0^{2\pi} \log \tilde{w}(e^{i(2\pi - \th)}) \, \frac{\rmd\th}{2\pi}\right) \\
		& =  \exp\left(\int_0^{2\pi} \log w(e^{i\th}) \, \frac{\rmd\th}{2\pi} \right) \cdot  \exp\left(\int_{2\pi}^{0} \log w(e^{i\th}) \, \frac{\rmd(-\th)}{2\pi} \right) \\
		& = \exp\left(\int_0^{2\pi} \log w(e^{i\th}) \, \frac{\rmd\th}{2\pi} \right)^2,
	\end{align*}
	and so, at last, we recover the classical result:
	\[
		\prod_{n=0}^{\infty} (1 - \lvert \g_n \rvert^2) = \exp\left(\int_0^{2\pi} \log w(e^{i\th}) \, \frac{\rmd\th}{2\pi} \right).
	\]

    We note that in this case, one correspondingly obtains \Cref{cor:SVConvergence} immediately via the classical argument.
\end{remark}

\section{Baxter's Theorem}
\label{sec:Baxter}

This paper concludes with a quaternionic Baxter's theorem. We shall need the following notion of absolute continuity for \(q\)-positive measures, generalising the usual notion in a natural fashion.

\begin{definition}
    Let \(i,j \in \SS\) be orthogonal and let \(\mu_i\) be a \(q\)-positive measure on \(\TT_i\). We say that \(\mu_i\) is \emph{absolutely continuous} (with respect to the normalised Lebesgue measure \(\frac{\rmd\th}{2\pi}\) on \(\TT_i\)) if there exists some \(w \in L^1(\TT_i; \HH)\) such that \(\rmd\mu_i(\th) = w(e^{i\th}) \frac{\rmd\th}{2\pi}\).
\end{definition}

We shall also require two notions of Wiener algebra previously studied in the literature. For a discussion on the theory of the matricial Wiener algebra, see Chapter 29, Sections 8 \& 9 of \cite{GGK93}. We shall denote the Wiener algebra on the circle \(\TT_i\) by \(\mathscr{W}_i\) and the \emph{positive} Wiener algebra on \(\TT_i\) by \(\mathscr{W}_{+,i}\). Details on the quaternionic Wiener algebra were established initially in \cite{ACKS16}; in keeping with this paper, we shall denote the quaternionic Wiener algebra by \(\mathscr{W}_{\HH}\).

\begin{theorem}
\label{thm:Baxter1}
    Let \(c = (c_n)_{n=-\infty}^{\infty}\) be a non-trivial positive semidefinite sequence, fix orthogonal \(i,j \in \SS\), and let \(\mu_i\) be the \(q\)-positive measure on \(\TT_i\) with moment sequence \(c\). If the Verblunsky coefficients \((\g_n)_{n=0}^{\infty}\) associated to \(c\) are summable, i.e. if 
    \[
        \sum_{n=0}^{\infty} \lvert \g_n \rvert < \infty,
    \]
    then \(\mu_i\) is absolutely continuous and the Radon--Nikodym derivative \(w\) of \(\mu_i\) extends uniquely to an element of the quaternionic Wiener algebra \(\mathscr{W}_i\).
\end{theorem}
\begin{proof}
    Suppose that \((\g_n)_{n=0}^{\infty}\) is summable and let \(M_i\) be the matrix-valued measure arising from \(\mu_i\); denote its Verblunsky coefficients by \((\a_n)_{n=0}^{\infty}\). We first observe that the norms of these Verblunsky coefficients are given by
    \begin{align*}
        \lVert \a_n \rVert =& \; \lVert \x_i(\g_n) \rVert = \max_{n=1,2} \sqrt{\l_n(\x_i(\g_n) \x_i(\g_n)^*)} \\ =& \; \max_{n=1,2} \sqrt{\l_n( \x_i(\lvert \g_n \rvert^2))} = \max_{n=1,2} \sqrt{\l_n\left(\begin{bmatrix} \lvert \g_n \rvert^2 & 0 \\ 0 & \lvert \g_n \rvert^2 \end{bmatrix}\right)} = \lvert \g_n \rvert,
    \end{align*}
    for any \(n \in \NN\). It follows that \((\g_n)_{n=0}^{\infty}\) is summable if and only if \((\a_n)_{n=0}^{\infty}\) is a summable sequence in \(\CC_i^{2\times 2}\). In this case, by the matrix Baxter's theorem first shown in \cite{Ger81} (see also \cite{DK16}), \(M_i\) is absolutely continuous, and its Radon--Nikodym derivative \(\D_i\) has
    \[
        \D_i = Q_i^* Q_i = R_i R_i^*
    \]
    where \(Q_i^{\pm 1}, R_i^{\pm 1} \in \mathscr{W}_{+,i}^{2\times2}\), the \(2\times 2\) matricial positive Wiener algebra on \(\TT_i\).
    
    Since \(\D_i\) takes values in positive semidefinite matrices, we may write
    \[
        \D_i = \begin{bmatrix}\D_{11} & \D_{12} \\[6pt] \overline{\D_{12}} & \D_{22} \end{bmatrix}.
    \]
    Now, as \(\rmd M_i = \D_i \frac{\rmd\th}{2\pi}\), its moments are given by
    \[
        \int_0^{2\pi} e^{i n \th} \D(e^{i\th}) \, \frac{\rmd\th}{2\pi} = \begin{bmatrix} \int_0^{2\pi} e^{i n \th} \D_{11}(e^{i\th}) \, \frac{\rmd\th}{2\pi} & \int_0^{2\pi} e^{i n \th} \D_{12}(e^{i\th}) \, \frac{\rmd\th}{2\pi} \\[6pt] \int_0^{2\pi} e^{i n \th} \overline{\D_{12}}(e^{i\th}) \, \frac{\rmd\th}{2\pi} & \int_0^{2\pi} e^{i n \th} \D_{22}(e^{i\th}) \, \frac{\rmd\th}{2\pi} \end{bmatrix};
    \]
    on the other hand, since \(M_i\) has moments \((\x_i(c_n))_{n=-\infty}^{\infty}\),
    \[
        \int_0^{2\pi} e^{i n \th} \D(e^{i\th}) \, \frac{\rmd\th}{2\pi} = \begin{bmatrix} \int_0^{2\pi} e^{i n \th} \, \rmd\mu_1(\th) & \int_0^{2\pi} e^{i n \th} \, \rmd\mu_2(\th) \\[6pt] -\int_0^{2\pi} e^{-i n \th} \, \rmd\overline{\mu_2}(\th) & \int_0^{2\pi} e^{-i n \th} \rmd\mu_1(\th)\end{bmatrix},
    \]
    where \(c = (c_n)_{n=-\infty}^{\infty}\) has \(q\)-positive representing measure \(\mu_i = \mu_i^{(1)} + \mu_i^{(2)} j\). It follows by comparing (1,1)- and (1,2)-entries that \(\mu^{(1)}, \mu^{(2)}\) are themselves both absolutely continuous, with densities
    \[
        \rmd\mu_i^{(1)} = \D_{11}(e^{i\th})\, \frac{\rmd\th}{2\pi}, \quad \quad \quad \rmd\mu_i^{(2)} = \D_{12}(e^{i\th})\,\frac{\rmd\th}{2\pi}.
    \]
    
    Writing \(Q_i = \begin{bmatrix}q_{11} & q_{12} \\ q_{21} & q_{22} \end{bmatrix}\), we have that \(q_{nm} \in \mathscr{W}_{+,i}\) for \(n,m=1,2\), and \(\D_i = Q_i^* Q_i\) implies that the density of \(\mu_1\) is \(\D_{11} = \lvert q_{11} \rvert^2 + \lvert q_{21} \rvert^2 \in \mathscr{W}_i\). We may similarly conclude that the density of \(\mu_2\) is \(\D_{12} = \overline{q_{11}}q_{12} + \overline{q_{21}}q_{22} \in \mathscr{W}_i\). Thus we may write \(\mu_i\) as
    \[
        \rmd\mu_i(\th) = \D_{11}(e^{i\th})\frac{\rmd\th}{2\pi} + \D_{12}(e^{i\th})\frac{\rmd\th}{2\pi}j = \left(\D_{11}(e^{i\th}) + \D_{12}(e^{i\th})j\right)\frac{\rmd\th}{2\pi},
    \]
    i.e. \(\mu_i\) is absolutely continuous with Radon--Nikodym derivative
    \[
        w(e^{i\th}) = \D_{11}(e^{i\th}) + \D_{12}(e^{i\th})j,
    \]
    where \(\D_{11}, \D_{12} \in \mathscr{W}_i\).

    It remains for us to identify \(w\) uniquely with an element of the quaternionic Wiener algebra \(\mathscr{W}_\HH\). To see this, we may use the fact that \(\D_{11}, \D_{12} \in \mathscr{W}_i\) to write
    \[
        \D_{11}(e^{i\th}) = \sum_{n=-\infty}^{\infty} e^{in\th} \D_{11}^{(n)}, \quad \D_{12}(e^{i\th}) = \sum_{n=-\infty}^{\infty} \D_{12}^{(n)}
    \]
    where the \(\CC\)-valued sequences \((\D_{11}^{(n)})_{n=-\infty}^{\infty}, (\D_{12}^{(n)})_{n=-\infty}^{\infty}\) are summable. Then we have for \(\th \in [0,2\pi)\) that
    \[
        w(e^{i\th}) = \left(\sum_{n=-\infty}^{\infty} e^{in\th} \D_{11}^{(n)}\right) + \left(\sum_{n=-\infty}^{\infty} e^{in\th} \D_{12}^{(n)}\right) j = \sum_{n=-\infty}^{\infty} e^{in\th} (\D_{11}^{(n)} + \D_{12}^{(n)}j)
    \]
    where the (now \(\HH\)-valued) coefficient sequence \((\D_{11}^{(n)} + \D_{12}^{(n)}j)_{n=-\infty}^{\infty}\) is also summable: if \(z,w \in \CC_i\) then the triangle inequality gives \(\lvert z + wj \rvert \leq \lvert z \rvert + \lvert w \rvert\) so that
    \[
        \sum_{n=-\infty}^{\infty} \lvert \D_{11}^{(n)} + \D_{12}^{(n)}j \rvert \leq \sum_{n=-\infty}^{\infty} \lvert \D_{11}^{(n)} \rvert + \lvert \D_{12}^{(n)} \rvert = \sum_{n=-\infty}^{\infty} \lvert \D_{11}^{(n)} \rvert + \sum_{n=-\infty}^{\infty} \lvert \D_{12}^{(n)} \rvert,
    \]
    which is bounded as the sum of two bounded quantities. Thus the function \(\tilde{w} : \partial\BB \to \HH\) given by
    \[
        \tilde{w}(p) = \sum_{n=-\infty}^{\infty} p^n (\D_{11}^{(n)} + \D_{12}^{(n)}j)
    \]
    is an element of the quaternionic Wiener algebra \(\mathscr{W}_\HH\) with \(\tilde{w}\vert_{\TT_i} = w\).

    Moreover, observe that this association is unique: if \(f \in \mathscr{W}_\HH\) has \(f\vert_{\TT_i} = w\) then \(f\vert_{\TT_i} - w \equiv 0\), so the Fourier coefficients \((f_n)_{n=-\infty}^{\infty}\) are equal to those of \(w\), and hence of \(\tilde{w}\), that is, \(f = w\) and the extension of \(w\) is unique.
\end{proof}

We conclude with the converse direction. We shall make use of the following notation.

\begin{definition}
	We define an involution on \(W_i\) via coefficient-wise conjugation: if \(w \in \mathscr{W}_i\) has Fourier series \(w(e^{i\th}) = \sum_{n=-\infty}^{\infty} e^{i n \th} w_n\), then
	\[
		\overline{w}(e^{i\th}) := \sum_{n=-\infty}^{\infty} e^{i n \th} \overline{w_n}.
	\]
\end{definition}

\begin{theorem}
    Let \(c = (c_n)_{n=-\infty}^{\infty}\) be a non-trivial positive definite sequence, fix orthogonal \(i,j\in\SS\), and let \(\mu_i\) be the \(q\)-positive measure on \(\TT_i\) with moment sequence \(c\). If \(\mu_i\) is absolutely continuous with Radon--Nikodym derivative \(w \in \mathscr{W}_\HH\), then the Verblunsky coefficients \((\g_n)_{n=0}^{\infty}\) associated to \(c\) are summable.
\end{theorem}

\begin{proof}
    Write \(\rmd\mu_i(\th) = w(e^{i\th}) \, \frac{\rmd\th}{2\pi} = (w^{(1)} + w^{(2)} j ) \, \frac{\rmd\th}{2\pi}\) with \(w \in \mathscr{W}_{\HH}\); then \(w^{(1)}\) and \(w^{(2)}\) both lie in \(\mathscr{W}_{i}\). Set \(C_n = \x_i(c_n)\) for \(n \in \ZZ\). It follows from \(q\)-positivity of \(\mu_i\) that for \(N = 1, 2, \ldots\) one has positivity of the \(N \times N\) block-Toeplitz matrices of \((C_n)_{n=-\infty}^{\infty}\):
    \[
        \begin{bmatrix} C_{m-n} \end{bmatrix}_{n,m=0}^{N} \succeq 0.
    \]
    
    The matrix-valued measure \(M_i\) associated to \(\mu_i\), i.e. the measure with moments \((C_n)_{n=-\infty}^{\infty}\), has
    \[
        C_n = \int_0^{2\pi} e^{in\th} \rmd M_i(\th),
    \]
    so using \(C_n = \x_i(c_n)\), if we write \(c_n = c_n^{(1)} + c_n^{(2)}j\) for \(n \in \NN\) and \(c_n^{(1)}, c_n^{(2)} \in \CC_i\), we have
    \[
        \int_0^{2\pi} e^{in\th} \rmd M_i(\th) = \begin{bmatrix} c_n^{(1)} & c_n^{(2)} \\[6pt] -\overline{c_n^{(2)}} & \overline{c_n^{(1)}} \end{bmatrix} = \begin{bmatrix} \int_0^{2\pi} e^{in\th} w^{(1)}(e^{i\th}) \frac{\rmd\th}{2\pi} & \int_0^{2\pi} e^{in\th}w^{(2)}(e^{i\th})\frac{\rmd\th}{2\pi} \\[6pt] - \int_0^{2\pi} e^{in\th} \overline{w^{(2)}}(e^{i\th})\frac{\rmd\th}{2\pi} & \int_0^{2\pi} e^{in\th}\overline{w^{(1)}}(e^{i\th}) \frac{\rmd\th}{2\pi}\end{bmatrix}. 
    \]
    Since this holds for all \(n\in\ZZ\), we may conclude that \(M_i\) has Radon--Nikodym derivative given by
    \[
        W_i(e^{i\th}) = \begin{bmatrix} w^{(1)}(e^{i\th}) & w^{(2)}(e^{i\th}) \\[6pt] -\overline{w^{(2)}}(e^{i\th}) & \overline{w^{(1)}}(e^{i\th}) \end{bmatrix}.
    \]
    
    Then
    \[
        \begin{bmatrix} c_n^{(1)} & c_n^{(2)} \\[6pt] -\overline{c_n^{(2)}} & \overline{c_n^{(1)}}\end{bmatrix} = C_n = \int_0^{2\pi} e^{in\th} \, \rmd M_i(\th) = \begin{bmatrix} \int_0^{2\pi} e^{in\th} \, \rmd M_i^{(1,1)} & \int_0^{2\pi} e^{in\th} \, \rmd M_i^{(1,2)} \\[6pt] \int_0^{2\pi} e^{in\th} \, \rmd M_i^{(2,1)} & \int_0^{2\pi} e^{in\th} \, \rmd M_i^{(2,2)} \end{bmatrix}.
    \]
    
    Since \(w \in \mathscr{W}_\HH\), we have since \(1\) and \(j\) are orthogonal that for \(\ell = 1,2\),
	\[
		\sum_{n=-\infty}^{\infty} \lvert w_n^{(\ell)} \rvert \leq \sum_{n=-\infty}^{\infty} \lvert w_n^{(1)} + w_n^{(2)}j \rvert = \sum_{n=-\infty}^{\infty} \lvert w_n \rvert < \infty,
	\]
	and hence \(w^{(1)}, w^{(2)} \in \mathscr{W}_i\), and accordingly the Radon--Nikodym derivative \(W_i\) of \(M_i\) lies in \(\mathscr{W}_i^{2\times2}\).

	By the matrix Baxter's theorem of \cite{DK16}, then, \(M_i\) has summable Verblunsky coefficients \((\a_n)_{n=0}^{\infty}\); by the discussion of \Cref{subsec:ConsistencyVCs} and the same calculation as in the proof of \Cref{thm:Baxter1}, we at last see that
	\[
		\sum_{n=0}^{\infty} \lvert \g_n \rvert = \sum_{n=0}^{\infty} \lVert\a_n\rVert < \infty.
	\]
\end{proof}

We complete the paper by combining the previous two theorems to state a full quaternionic analogue of Baxter's theorem.

\begin{corollary}
\label{cor:QuaternionicBaxter}
    Let \(c = (c_n)_{n=-\infty}^{\infty}\) be a non-trivial positive definite sequence, fix orthogonal \(i,j \in \SS\) and let \(\mu_i\) be the \(q\)-positive measure on \(\TT_i\) with moment sequence \(c\). Then the Verblunsky coefficients of \(c\) are summable if and only if \(w\) is absolutely continuous and its Radon--Nikodym derivative \(w\) is the restriction to \(\TT_i\) of a quaternionic Wiener algebra function.
\end{corollary}

\nocite{Goh86}
\bibliographystyle{plain}

\end{document}